\documentclass[a4paper,11pt]{scrartcl}
\usepackage{pgfplots}
\usepackage[utf8]{inputenc}
\usepackage{lmodern}
\usepackage[english]{babel}
\usepackage[unicode, pdftex,
            pdfauthor={Wodson Mendson},
            pdftitle={Foliations on smooth algebraic surfaces over positive characteristic},
            pdfsubject={foliations, algebraic geometry},
            pdfkeywords={foliations on surfaces, p-divisor},
            pdfproducer={Latex with hyperref},
            pdfcreator={Overleaf}]{hyperref}
\usepackage{amsmath}
\usepackage{amsfonts}
\usepackage{amssymb}
\usepackage{amsthm}
\usepackage{MnSymbol}
\usepackage{titling}
\usepackage{enumerate}
\DeclareOldFontCommand{\sc}{\normalfont\scshape}{\@nomath\sc}
\newtheorem*{resumo}{Abstract}

\newtheorem{propp}{Proposition}
\newtheorem*{thmA}{Theorem A}
\newtheorem*{thmB}{Theorem B}

\newtheorem*{thmJou}{Theorem (Jouanolou)}
\newtheorem{theorem}{Theorem}

\newtheorem{pro}{Problem}
\newtheorem{thm}{Theorem}[section]
\newtheorem{lemma}[thm]{Lemma}
\newtheorem{cor}[thm]{Corollary}
\newtheorem{prop}[thm]{Proposition}
\newtheorem{OBS}[thm]{Remark}
\newtheorem{ex}[thm]{Example}
\newtheorem{dfn}[thm]{Definition}

\DeclareMathOperator{\pt}{pt}
\DeclareMathOperator{\corpo}{k}
\DeclareMathOperator{\Der}{Der}
\DeclareMathOperator{\mdc}{gcd}
\DeclareMathOperator{\codim}{codim}
\DeclareMathOperator{\sing}{sing}

\DeclareMathOperator{\od}{ord}
\DeclareMathOperator{\Pic}{Pic}

\DeclareMathOperator{\Div}{Div}

\DeclareMathOperator{\Num}{Num}
\DeclareMathOperator{\h}{H^{0}}

\DeclareMathOperator{\spm}{Spm}

\DeclareMathOperator{\Projdois}{\mathbb{P}_{k}^{2}}
\DeclareMathOperator{\Projum}{\mathbb{P}_{k}^{1}}
\DeclareMathOperator{\Projumprod}{\mathbb{P}_{k}^{1}\times\mathbb{P}_{k}^{1}}

\usepackage{authblk}

\title{Foliations on smooth algebraic surfaces over positive characteristic}
\author{Wodson  Mendson}

\AtEndDocument{\bigskip{\footnotesize%
  \textsc{Instituto de Matemática Pura e Aplicada, Estrada Dona Castorina 110, 22460-320 Rio de Janeiro RJ, Brazil} \par  
  \textit{E-mail address}: \texttt{oliveirawodson@gmail.com} \par
  \addvspace{\medskipamount}
}}

\makeatletter
\def\keywords{\xdef\@thefnmark{}\@footnotetext}
\makeatother

\begin{document}



\maketitle
     \footnotetext{This work was developed at IMPA. W. Mendson acknowledges support of CNPq and Faperj}

\begin{resumo} \normalfont  We investigate the notion of the $p$-divisor for foliations on a smooth algebraic surface defined over a field of positive characteristic $p$ and we study some of their properties. We present a structure theorem for the $p$-divisor of foliations in the projective plane and the Hirzebruch surfaces where we show that, under certain conditions, such $p$-divisors are reduced.
\end{resumo}

\setcounter{tocdepth}{1}
\tableofcontents

\section{Introduction}

In his  monograph \cite{MR537038} Jouanolou showed that a very generic foliation of degree $d\geq2$ in the projective plane has no algebraic solutions. The crucial point of his argument consists in constructing examples, for each degree $d\geq2$, of foliations with no algebraic solutions. The term very generic means that there exists a countable union $F$ of closed sets in the space of holomorphic foliations of degree $d$ such that for any foliations $\mathcal{F}$ lying outside $F$ has no algebraic solutions.

In order to prove his theorem, Jouanolou showed the following result.

\begin{thmJou} For every $d \in \mathbb{Z}_{\geq 2}$ the foliation in $\mathbb{P}_{\mathbb{C}}^{2}$ defined by the vector field
     $$
        v_{d} = (xy^{d}-1)\frac{\partial}{\partial x}-(x^{d}-y^{d+1})\frac{\partial}{\partial y}
    $$
has no algebraic solutions.
\end{thmJou}

Recently, a similar problem was considered in \cite{lizarbe2021density} for Hirzebruch surfaces. By assuming certain conditions in the normal bundle it is shown that a foliation in the Hirzebruch surfaces has no algebraic solutions (see \cite[Theorem C]{lizarbe2021density}). 

Interestingly, the analogue of the Jouanolou's theorem over a field of positive characteristic is completely false (see \cite{MR1904485}). The results established in \cite{MR1904485} imply the following proposition.

\begin{propp}\label{JVP} Let $\corpo$ be an algebraically closed field of characteristic $p>0$. Let $\mathcal{F}$ be a foliation on $\mathbb{P}_{\corpo}^{2}$ and suppose that $\deg(\mathcal{F})<p-2$. Then $\mathcal{F}$ has an invariant algebraic curve.
\end{propp}

In positive characteristic $p$ there are two classes of foliations: the foliations that are $p$-closed and the foliations that are not. Given a foliation $\mathcal{F}$ in a smooth algebraic surface $X$ defined over a field $\corpo$ of characteristic $p>0$, we say that it is $p$-closed if their tangent sheaf $T_{\mathcal{F}}$ is closed by $p$-powers. This is equivalent to say that the $\mathcal{O}_{X}$-morphism 
$$
    \psi_{\mathcal{F}} \colon F_{X}^{*}T_{\mathcal{F}}\longrightarrow N_{\mathcal{F}}\qquad v \mapsto v^{p}\mod T_{\mathcal{F}} 
$$
is the zero  morphism where $F_{X}$ is the absolute Frobenius morphism and $N_{\mathcal{F}}$ is the normal bundle of $\mathcal{F}$. The morphism $\psi_{\mathcal{F}}$ is called the \textbf{$p$-curvature} of the foliation $\mathcal{F}$.

When the foliation is $p$-closed it follows from \cite[Théorème 1]{Brunella1999SurLH} that there are infinitely many algebraic solutions. On the other hand, if the foliation $\mathcal{F}$ is not $p$-closed then there is a divisor $\Delta_{\mathcal{F}}$, the \textbf{$p$-divisor}, which is defined as the degeneracy locus of the $p$-curvature morphism $\psi_{\mathcal{F}}$. An interesting property of the $p$-divisor is that every irreducible algebraic solution of the foliations is contained in the support of the $p$-divisor.

Since the set of foliations that are not $p$-closed is open we can ask about the structure of the $p$-divisor for generic foliations. 

\begin{pro} Let $X$ be a smooth algebraic surface defined over a field $\corpo$ of characteristic $p>0$. What can we say about the $p$-divisor $\Delta_{\mathcal{F}}$ of a generic foliation $\mathcal{F}$? 
\end{pro}

We present results towards the solution of this problem in the case where $X$ is the projective plane or a Hirzebruch surface. More precisely, we show the following theorems.

\begin{theorem}\label{casoplano} Let $\corpo$ be an algebraically closed field of characteristic $p>0$. A generic foliation in the projective plane $\mathbb{P}_{\corpo}^{2}$ of degree $d\geq 1$ with $p\nmid d-1$ has reduced $p$-divisor.
\end{theorem}

\begin{theorem}\label{Hirz} Let $\corpo$ be an algebraically closed field of characteristic $p>0$. Let $\Sigma_{d}$ be the $d$-Hirzebruch surface defined over $\corpo$ and let $d_{1},d_{2} \in \mathbb{Z}_{\geq 0}$ such that $p\nmid d_{i}$, if $d_i\neq 0$. Let $F$ be a fiber of the natural projection  $\pi\colon \Sigma_{d} \longrightarrow \mathbb{P}_{\corpo}^{1}$ and $M_{d}$ be a section which satisfies  $ F\cdot M_{d} = 1$ and $ M_{d}^{2} = d$. Then a generic foliation in  $\Sigma_{d}$ with normal bundle $N$ which is numerically equivalent to $(d_{1}-d+2)F+(d_{2}+2)M_{d}$ has reduced $p$-divisor. 
\end{theorem}

\subsection{Organization of the paper} In Section \ref{notacoes} we fix the notations that will be used in the paper. In Section \ref{folheacoesemsuperficies}, we introduce the definition of $p$-divisor of foliations on algebraic surfaces defined  over a field of characteristic $p>0$. We recall the global representation of foliations in the projective plane and in the Hirzebruch surface. We present the definition of the $p$-divisor of a foliation and we investigate some of their properties, in particular, some applications to algebraicity of foliations on complex projective plane. In Section \ref{pdivisoremP2} we consider the problem of the structure of the $p$-divisor for foliations in the projective plane and in the Section \ref{pdivisorP1P1} we study the  problem for  $\mathbb{P}_{\corpo}^{1}\times\mathbb{P}_{\corpo}^{1}$. In Section \ref{pdivisorSigma} we finalize by considering  the Hirzebruch surface of type $d>0$.
\section{Notation} \label{notacoes}

\begin{itemize}
    \item $\corpo = $ algebraically closed field of characteristic $p>0$.
    
    \item $X = $ smooth algebraic surface defined over $\corpo$.
    
    \item Curve in $X$ $=$ effective divisor in $X$.
    
    \item $d$-Hirzebruch surface over $\corpo$ ($d\geq 0$)  $=$ $\Sigma_{d} = \mathbb{P}(\mathcal{O}_{\mathbb{P}_{\corpo}^{1}}\oplus 
    \mathcal{O}_{\mathbb{P}_{\corpo}^{1}}(d))$.
    
    \item $M = $ the curve in $\Sigma_{d}$ ($d>0$) such that $M^{2} = -d$.
    
    \item $\equiv$  $=$ numerical equivalence in $\Div(X)$.
    
    \item $\Num_{\mathbb{Q}}(X) = (\Div(X)/\equiv)\otimes
\mathbb{Q}$.

    \item If $(X,H)$ is a polarized surface and  $D$ is a divisor in $X$, $\deg(D) = D\cdot H$.
    
    \item $U(X,H,N), V(X,H,\mathcal{N}) = $ open sets of Lemma \ref{aberto}.
    
    \item $m_{Q}(C) =$ algebraic multiplicity of the curve $C \subset X$ at $Q \in X$.
    
    \item If $\pi_{Q}\colon Bl_{Q}(X)\longrightarrow X$ is the blow up at $Q$, $\mathcal{F}$ is a foliation in $X$, $\mathcal{G} = \pi_{Q}^{*}\mathcal{F}$ and $E$ is the exceptional divisor,  $l(Q) = \od_{E}(N_{\mathcal{G}}^{*}-\pi_{Q}^{*}N_{\mathcal{F}}^{*})$.
    
    \item $O(2) = $ terms of the order at least $2$.
    
\end{itemize}

\section{Foliations on surfaces and the $p$-divisor} \label{folheacoesemsuperficies}

Let $\corpo$ be an algebraically closed field and $X$ be a smooth algebraic surface defined over $\corpo$. A \textbf{foliation} $\mathcal{F}$ in $X$ is a coherent subsheaf  $T_{\mathcal{F}}\subset T_{X}$ of rank one which satisfies the following properties:
\begin{itemize}
    \item  The sheaf $T_{\mathcal{F}}$
 is closed by Lie brackets;    
    \item The  quotient $T_{X}/T_{\mathcal{F}}$ is torsion free, that is, $T_{\mathcal{F}}$ is saturated in $T_{X}$.
\end{itemize}
We can define a foliation in $X$ in more explicitly terms. In this terms, a foliation in $X$ consists of a system $\{(U_{i},\omega_{i},v_{i})\}_{i\in I}$ such that:

\begin{itemize}
    \item The collection $\{U_{i}\}_{i\in I}$ is an open cover of $X$;
    \item For each $i \in I$ we have $v_{i} \in T_{X}(U_{i})$, $\omega_{i} \in \Omega_{X/k}^{1}(U_{i})$ such that $i_{v_{i}}\omega_{i} = 0$;
    \item In $U_{i}\cap U_{j}$ we have $\omega_{i} = f_{ij}\omega_{j}$ and $v_{i} = g_{ij}v_{j}$ for some functions $f_{ij},g_{ij} \in \mathcal{O}_{X}(U_{ij})^{*}$;
    \item For each $i\in I$ we have $\codim \sing(\omega_{i})\geq 2$ and $\codim \sing(v_{i})\geq 2$.
\end{itemize}

Note that the second definition is an alternative version of the first. Indeed, given a foliation  $\{(U_{i},\omega_{i},v_{i})\}_{i\in I}$ in $X$ we can construct a saturated subsheaf of $T_{X}$ in the following way: For each open $U$ of $X$ we define $T_{\mathcal{F}}(U)$ by
$$
    T_{\mathcal{F}}(U) = \{v \in T_{X}(U)\mid i_{v}\omega_i|_{U_i} = 0 \mbox{ in $U\cap U_i$ for all $i\in I$}\}.
$$
We ensure that $T_{\mathcal{F}}$ is a saturated subsheaf of $T_{X}$ by using the condition imposed in the singular set of $\omega_i$.  Reciprocally, let $T_{\mathcal{F}}$ be a foliation in $X$ and consider the global section  $\omega \in \h(X,\Omega_{X}^{1}\otimes N_\mathcal{F})$ induced by the morphism $T_{X} \longrightarrow N_{\mathcal{F}}$. Then, by definition we obtains a open cover $\{U_{i}\}_{i\in I}$  of  $X$, $1$-forms $\omega_{i}\in \Omega_{X}^{1}(U_i)$
and functions $\{f_{ij}\}_{i,j}$ with $f_{ij} \in \mathcal{O}_{X}^{*}(U_{i}\cap U_{j})$ representing $N_\mathcal{F}^{*}$ such that $\omega_{i}= f_{ij}\omega_{j}$. Since $T_{\mathcal{F}}$ is a saturated subsheaf of $T_{X}$ we ensure that  $\codim \sing (\omega_i) \geq 2$ for every $i$. Note that the vector fields $v_{i}$ are obtained in a similar way by considering a global section $v \in \h(X,T_{X}\otimes T_{\mathcal{F}}^{*})$ induced by the inclusion $T_{\mathcal{F}}\subset T_{X}$. By construction, we obtain
vector fields $\{v_{i}\}_{i\in I}$ in $\{U_{i}\}_{i\in I}$ such that for every $i$ the vector field $v_{i}$ is tangent to the $1$-form $\omega_i$ that defines $\mathcal{F}$ in the open set $U_{i}$.

Let  $\{(U_{i},\omega_{i},v_{i})\}_{i\in I}$ be a foliation in $X$. The collection
$\{f_{ij}^{-1}\}, \{g_{ij}\}$ determines elements of $\h(X,\mathcal{O}_{X}^{*}) = \Pic(X)$ and the line bundles associated are the \textbf{conormal} bundle $\Omega^{1}_{X/\mathcal{F}}$ and the \textbf{cotangent} bundle $\Omega^{1}_{\mathcal{F}}$ associated to $\mathcal{F}$. Any divisor in the correspondent linear classes to $\Omega_{\mathcal{F}}^{1}$ and $(\Omega^{1}_{X/\mathcal{F}})^{*}$ will be called the \textbf{canonical divisor} and the \textbf{normal divisor} associated to $\mathcal{F}$ and will be denoted by $K_{\mathcal{F}}$ and $N_{\mathcal{F}}$.

\subsection{The \texorpdfstring{$p$}{p}-divisor for foliations on smooth algebraic surfaces}

We start this section recalling the following basic lemma about derivations over fields of positive characteristic.

\begin{lemma}\label{Basicao} Let $\corpo$ be an field of characteristic $p>0$ and $R$ be a $\corpo$-domain. Let  $D \in \Der_{k}(R)$ be a $\corpo$-derivation. Then $D^{p} \in \Der_{\corpo}(R)$. If $f \in R$ then
    $$
        (fD)^{p} = f^{p}D^{p}-fD^{p-1}(f^{p-1})D.
    $$
\end{lemma}
\begin{proof} \cite[Proposition (5.3)]{katz1}.
\end{proof}

Let $X$ be an algebraic surface defined over a algebraically closed field $\corpo$ of characteristic $p>0$. Let $\mathcal{F} = \{(U_{i},\omega_{i},v_{i})\}$ be a foliation on $X$.

\begin{dfn} We say that $\mathcal{F}$ is \textbf{$p$-closed} if for some $i$ we have $v_{i}\wedge v_{i}^{p}= 0$.
\end{dfn}

\begin{OBS} By Lemma \ref{Basicao} we conclude that $v_{i}\wedge v_{i}^{p} = 0$ for some  $i$ if and only if $v_{j}\wedge v_{j}^{p} = 0$ for every $j$.
\end{OBS}

Suppose that $\mathcal{F}$ is not $p$-closed. In the open $U_{ij} =  U_{i}\cap U_{j}$ we have $\omega_{i} = f_{ij}\omega_{j}$ and $v_{i} = g_{ij}v_{j}$. Since we are assuming that $\mathcal{F}$ is not  $p$-closed we have for each  $i,j \in I$
    $$
        0 \neq i_{v_{i}^{p}}\omega_{i} = i_{(g_{ij}v_{j})^{p}}f_{ij}\omega_{j} = i_{(g_{ij}^{p}v_{j}^{p}+g_{ij}v_{j}^{p-1}(g_{ij}^{p-1})v_{j})}f_{ij}\omega_{j} = g_{ij}^{p}f_{ij}i_{v_{j}^{p}}\omega_{j}\neq 0.
    $$
So, the collection  $\{i_{v_{i}^{p}}\omega_{i}\}_{i\in I}$ determines a global section $0\neq s_{\mathcal{F}} \in H^{0}(X,(\Omega_{\mathcal{F}}^{1})^{\otimes p}\otimes N_{\mathcal{F}})$. 

\begin{dfn} Let $\mathcal{F}$ be a foliation in $X$ that is not $p$-closed. The \textbf{$p$-divisor} associated to $\mathcal{F}$ is the zero divisor of the section $s_{\mathcal{F}}$, that is,$
\Delta_{\mathcal{F}} = (s_{\mathcal{F}})_{0} \in \Div(X).$
\end{dfn}

\begin{OBS} Let $\mathcal{F}$ be a foliation in $\Projdois$ of degree $d>0$. Then, we have the formulas: $\Omega_{\mathcal{F}}^{1} = \mathcal{O}_{\Projdois}(d-1)$ and $N_{\mathcal{F}} = \mathcal{O}_{\Projdois}(d+2)$. In particular, if $\mathcal{F}$ is not $p$-closed then the $p$-divisor is a divisor of degree $\deg(\Delta_{\mathcal{F}}) = p(d-1)+d+2$.
\end{OBS}

\begin{ex} Let $X = \mathbb{A}_{\corpo}^{2}$ and $\alpha \in \corpo$. Let  $\mathcal{F}$ be a foliation defined by the vector field
$
v = x\partial_{x}+\alpha y\partial_{y}.
$
Then, $\mathcal{F}$ is $p$-closed if and only if $\alpha \in\mathbb{F}_{p}$ and if $\alpha\not\in \mathbb{F}_{p}$ we have $s_{\mathcal{F}} = (\alpha^{p}-\alpha)xy$.
\end{ex}

\begin{prop} \label{inv} Let $X$ be smooth algebraic surface defined over $\corpo$ and $\mathcal{F}$ be a foliation in $X$ that is not $p$-closed. Let $C$ be an irreducible algebraic curve on $X$. If $C$ is $\mathcal{F}$-invariant then $\od_{C}(\Delta_{\mathcal{F}})>0$. Reciprocally, if $p$ does not divides $\od_{C}(\Delta_{\mathcal{F}})$ then $C$ is $\mathcal{F}$-invariant. 
\end{prop}

\begin{proof} Suppose that $C$ is $\mathcal{F}$-invariant and let $R = \mathcal{O}_{X,C}$ be the ring of regular functions of $X$ along $C$. Let $U$ be an affine open set such that $T_{\mathcal{F}}$ if given by a regular vector field $v$ and  $N_{\mathcal{F}}^{*}$ is given by a regular $1$-form $\omega$.  Let $\{f=  0\}$ be the local equation for
$C$ in $U$ and note that $f$ is an uniformizer parameter
to the ring $R$. We need to show that  $\od_{f}(\Delta_{\mathcal{F}})>0$. Now, we have
$v(f) = fH$ and $\omega \wedge df = f\sigma
$
for some $H \in R$ and $\sigma \in \Omega_{R/\corpo}^{2}$. Contraction with the vector field $v^{p}$ gives
    $
        fi_{v^{p}}(\sigma) = i_{v^{p}}(\omega\wedge df) = i_{v^{p}}\omega df-i_{v^{p}}(df)\omega.
    $
By using the equality $v(f) = fH$ we conclude that $v^{p}(f) = fH_{p}$ for some regular function $H_p$.  So,
$
fi_{v^{p}}\sigma+v^{p}(f)\omega = i_{v^{p}}\omega df$ and we conclude that $i_{v^{p}}\omega \in \langle f \rangle
$. It follows that $\od_{C}(\Delta_{\mathcal{F}})>0$.

Reciprocally, suppose that $\od_{C}(\Delta_{\mathcal{F}}) = \alpha \not\equiv 0 \mod p$ and write $\Delta_{\mathcal{F}} = f^{\alpha}g$ with $g \in R^{*}$. By \cite[Theorem 6.2]{MR2324555} we know that $d(\Delta_{\mathcal{F}}^{p-1}\omega) = 0$ and expanding this formula we obtain
$
\alpha\cdot gdf\wedge \omega  = f(gd\omega-dg\wedge \omega),
$
which implies $f|df\wedge \omega$. So, $C$ is $\mathcal{F}$-invariant.
\end{proof}

The Proposition \ref{inv} has the following consequence (compare with Proposition \ref{JVP}).

\begin{cor} Let $\corpo$ be an algebraically closed field of characteristic $p>0$. Let $\mathcal{F}$ be a foliation on $\mathbb{P}_{\corpo}^{2}$ and suppose that $p\nmid d+2$. Then $\mathcal{F}$ has an invariant algebraic curve.
\end{cor}

\begin{proof}  If $\mathcal{F}$ is $p$-closed then $\mathcal{F}$ admits infinitely many invariant curves by \cite[Théorème 1]{Brunella1999SurLH}. We can assume that $\mathcal{F}$ is not $p$-closed. Since $p$ does not divide $d+2$ the degree formula: $\deg(\Delta_{\mathcal{F}}) = p(d-1)+d+2$ shows that  $\deg(\Delta_{\mathcal{F}}) \not\equiv 0 \mod p$. In particular, there exists a prime divisor $\mathfrak{P}$ in the support of $\Delta_{\mathcal{F}}$ such that $\od_{\mathfrak{P}}(\Delta_{\mathcal{F}}) \not\equiv 0 \mod p$. By Proposition \ref{inv}, the divisor $\mathfrak{P}$ defines a $\mathcal{F}$-invariant irreducible curve.
\end{proof}

Let $\mathcal{F}$ be a foliation in a smooth algebraic surface $X$ and $Q \in \sing(\mathcal{F})$. We say that $Q$ is \textbf{not degenerated} with \textbf{eigenvalue} $\alpha \in \corpo$ if there exists an affine open subset $U\subset \mathbb{A}_{\corpo}^{2}$ which contains $Q$ such that $\mathcal{F}|_{U}$ can be represent by a polynomial vector field  $v = v_{1}+O(2)$ with $v_1 = x\partial_{x}+\alpha y\partial_{y}$ and $\alpha\neq0$.

\begin{dfn} Let $\mathcal{F}$ be a foliation on  $X$ and $Q \in \sing(\mathcal{F})$. We say that $Q$ is $p$-\textbf{reduced} if $Q$ is not degenerated and has type $\alpha(Q)$ satisfying the condition: $\alpha(Q) \not\in \mathbb{F}_{p}$.
\end{dfn}
 
 \begin{lemma}\label{intlocal} Let $\mathcal{F}$ be a foliation in a algebraic surface $X$ and $Q \in \sing(\mathcal{F})$ be a $p$-reduced singularity of $\mathcal{F}$. Then, $\mathcal{F}$ is not $p$-closed.
\end{lemma}
\begin{proof} Let $U$ be an affine open set that contains $Q$ and $x,y \in \mathcal{O}_{X,Q}$ local parameters system at $Q$ in $U$. In the open $U$ the foliation is given by a vector field $v = v_{1}+\tilde{v}$ where $v_{1} = x\partial_{x}+\alpha y\partial_{y}$ with $\alpha \not\in \mathbb{F}_{p}$ and $\tilde{v}$ consists of terms that has order at least two. Then,
$v^{p} = v_{1}^{p}+\tilde{v}_{p}$ where $\tilde{v}_{p}$
contains only homogeneous terms with order at least two. Observe that $v_{1}^{p} = x\partial_{x}+\alpha^{p} y\partial_{y}$ and $v_{1}\wedge v_{1}^{p}$ is the homogeneous component of the smallest degree that occurs in $v\wedge v^{p}$. Since $\alpha \not\in \mathbb{F}_{p}$ we ensure that $v_{1}\wedge v_{1}^{p} = (\alpha^{p}-\alpha)\partial_{x}\wedge \partial_{y}\neq 0$ and so $v\wedge v^{p}\neq0$. 
 \end{proof}

\begin{lemma}\label{explosao} Let $X$ be a smooth projective surface  defined over $\corpo$ and $\mathcal{F}$ be a foliation in $X$ that is not $p$-closed. Let $Q \in \sing(\mathcal{F})$ be a $p$-reduced singular point of $\mathcal{F}$. Suppose that  $\Delta_{\mathcal{F}}$ is a reduced divisor and let $\pi_{Q}: Bl_{Q}(X) \longrightarrow X$ the blowup with center at $Q$. Then, $\pi^{*}\mathcal{F}$ defines a foliation in $Bl_{Q}(X)$ with reduced $p$-divisor.
\end{lemma}

\begin{proof} Let $\mathcal{G} = \pi_{Q}^{*}\mathcal{F}$ be a foliation induced in $Bl_{Q}(X)$. Since $Q$ is $p$-reduced we have $ K_{\mathcal{G}}-\pi_{Q}^{*}K_{\mathcal{F}} = 0$ and $N_{\mathcal{G}} -\pi_{Q}^{*}N_{\mathcal{F}} = -E.
$
Indeed, in an affine open set $U\subset \mathbb{A}_{\corpo}^{2}$ the foliation is represented by a $1$-form $\omega = \omega_{1}+O(2)$ where $\omega_{1} = ydx-\alpha xdy$ with $\alpha \not \in \mathbb{F}_{p}$ and $O(2)$ containing only terms with order at least two. In a convenient coordinates system the map $\pi_{Q}:Bl_{Q}(X)\longrightarrow X$ associate $(x,t)\mapsto (x,xt)$. Since $\pi_{Q}^{*}\omega$ is a local section of  $N^{*}_{\pi_{Q}^{*}\mathcal{F}}$ we have $[\pi_{Q}^{*}\omega]_0 = N^{*}_{\pi_{Q}^{*}\mathcal{F}}$. In the other hand, since $Q$ is $p$-reduced we have $$\pi_{Q}^{*}\omega = \pi_{Q}^{*}\omega_{1}+O(2) = (xtdx-\alpha x(xdt+tdx))+O(2) = x(t(1-\alpha)dx+xdt+O(2)).$$ Denote $\tilde{\omega} = t(1-\alpha)dx+xdt+O(2)$ and note that $\tilde{\omega}$ is a local section of  $\pi_{Q}^{*}N_{\mathcal{F}}$. So, $N_{\mathcal{G}}^{*} = N^{*}_{\pi_{Q}^{*}\mathcal{F}} = [\pi_{Q}^{*}\omega]_0 = [x]_0+[\tilde{\omega}]_0  = E+\pi_{Q}^{*}N_{\mathcal{F}}^{*}$. The formula that compares $\pi_{Q}^{*}K_{\mathcal{F}}$ and $K_{\pi_{Q}^{*}\mathcal{F}}$ follows by the adjunction formula: $K_{X} = K_{\mathcal{F}}-N_{\mathcal{F}}$.
So, we have 
$$
[\Delta_{\mathcal{G}}] = pK_{\mathcal{G}}+N_{\mathcal{G}} = p\pi_{Q}^{*}K_{\mathcal{F}}+\pi_{Q}^{*}N_{\mathcal{F}}-E = \pi_{Q}^{*}[\Delta_{\mathcal{F}}] -E = \Tilde{\Delta}_{\mathcal{F}}+(m_{Q}(\Delta_{\mathcal{F}})-1)E
$$
where $\Tilde{\Delta}_{\mathcal{F}}$ denotes the strict transform of the divisor $\Delta_{\mathcal{F}}$ and $m_{Q}(\Delta_{\mathcal{F}})$ is the algebraic multiplicity of $\Delta_{\mathcal{F}}$ at $Q$. Since $Q$ is $p$-reduced, by using \cite[Fact 2.8]{MR3687427} we conclude that $m_{Q}(\Delta_{\mathcal{F}}) = 2$ so that $[\Delta_{\mathcal{G}}] = \Tilde{\Delta}_{\mathcal{F}}+E$.
\end{proof}

\begin{lemma}\label{aberto} Let $(X,H)$ be a polarized smooth  projective  surface defined over  algebraically closed field $\corpo$ with characteristic $p>2$. Let $\mathcal{N}$ be an invertible sheaf and let $\mathbb{F}ol_{\mathcal{N}}(X)  =\mathbb{P}(H^{0}(X,\Omega_{X}^{1}\otimes \mathcal{N}))$ the space of foliations on $X$ that has normal bundle $\mathcal{N}$. Suppose that that space is not empty and consider the following sets 
$$U= \{\mathcal{F} \in\mathbb{F}ol_{\mathcal{N}}(X)  \mid \Delta_{\mathcal{F}} \mbox{ is reduced } \} \quad \mbox{and}\quad V= \{\mathcal{F} \in\mathbb{F}ol_{\mathcal{N}}(X)  \mid \Delta_{\mathcal{F}} \mbox{ is prime} \}.$$
Then, $U$ and $V$ are open sets.
\end{lemma}

\begin{proof} First, note that $\deg(\Delta_{\mathcal{F}})$ depends only of $X$ and $\mathcal{N}$. Indeed, we have the following formula:
$$
\deg(\Delta_{\mathcal{F}}) = pK_{\mathcal{F}}\cdot H+\mathcal{N}\cdot H = p(K_{X}+\mathcal{N})\cdot H+\mathcal{N}\cdot H = pK_{X}\cdot H+(p+1)\mathcal{N}\cdot H.
$$
Given $e \in \mathbb{Z}_{\geq1}$ let
$Z_{e}(X)$ the space that consists of all curves in
$X$ that has degree $e$. In the following, we will use the fact that $Z_{e}(X)$ is a projective algebraic variety  over  $\corpo$ (see \cite[Theorem 1.4]{MR1440180}). Define the following sets
$$
S_{e}= \{(C,\mathcal{F}) \in Z_{e}(X)\times \mathbb{F}ol_{\mathcal{N}}(X)\mid \mbox{$2C\leq \Delta_{\mathcal{F}}$} \},
$$
$$
\Tilde{S}_{e} = \{(C,\mathcal{F}) \in Z_{e}(X)\times \mathbb{F}ol_{\mathcal{N}}(X)\mid \mbox{$C\leq \Delta_{\mathcal{F}}$}\}.
$$
Since the conditions $2C\leq \Delta_{\mathcal{F}}$ and $C\leq \Delta_{\mathcal{F}}$ are closed relations we have that  $S_{e}$ and $\Tilde{S}_{e}$ are closed sets in $Z_{e}(X)\times \mathbb{F}ol_{\mathcal{N}}(X)$. Let $\pi_{e}: Z_{e}(X)\times \mathbb{F}ol_{\mathcal{N}}(X) \longrightarrow \mathbb{F}ol_{\mathcal{N}}(X)$ the natural projection. Since $\pi_{e}$ is a proper morphism we ensure that $\pi_{e}(S_{e})$ and $\pi_{e}(\Tilde{S}_{e})$ are closed sets in $\mathbb{F}ol_{\mathcal{N}}(X)$. Denote by $\mathbb{F} \subset \mathbb{F}ol_{\mathcal{N}}(X)$ the closed set consisting of $p$-closed foliations and consider the following sets:
$$
T_{1} = \mathbb{F}ol_{\mathcal{N}}(X)-(\mathbb{F}\cup\bigcup_{j=1}^{\lfloor \frac{\deg(\Delta_{\mathcal{F}})}{2}\rfloor}\pi_{j}(S_{j})) \qquad\mbox{and} \qquad T_{2} = \mathbb{F}ol_{\mathcal{N}}(X)-(\mathbb{F}\cup\bigcup_{j=1}^{\deg(\Delta_{\mathcal{F}})-1}\pi_{j}(\Tilde{S}_{j})).
$$
We claim that the following identities holds: $U = T_{1}$ and $V = T_{2}$. Indeed, the inclusions
$U \subset T_{1}$ and $V \subset T_{2}$ are trivial. Now, let  $\mathcal{G}_{1} \in T_{1}$ and $\mathcal{G}_{2} \in T_{2}$.  Suppose, by contradiction, that  $\Delta_{\mathcal{G}_{1}}$ is not reduced and that $\Delta_{\mathcal{G}_{2}}$ is not a prime divisor. In particular, there are curves $C_{1}$ and $C_{2}$ in $X$ such that $2C_{1}\leq \Delta_{\mathcal{G}_{1}}$ and $C_{2}\leq \Delta_{\mathcal{G}_{2}}$ with $\deg(C_{2})<\deg(\Delta_{\mathcal{G}_{2}})$. Since $\deg(\Delta_{\mathcal{G}_{i}})$ depends only on $K_{X}$ and $\mathcal{N}$ by computing the degrees we have:
$$
2\deg(C_{1})\leq \deg(\Delta_{\mathcal{G}_{1}}) = \deg(\Delta_{\mathcal{F}}) \qquad \mbox{and}\qquad \deg(C_{2})\leq \deg(\Delta_{\mathcal{G}_{2}})-1 =\deg(\Delta_{\mathcal{F}})-1
$$
and this implies
$$\mathcal{G}_{1}  \in \mathbb{F}\cup\bigcup_{j=1}^{\lfloor \frac{\deg(\Delta_{\mathcal{F}})}{2}\rfloor}\pi_{j}(S_{j}) \qquad \mbox{and} \qquad  \mathcal{G}_{2} \in \mathbb{F}\cup\bigcup_{j=1}^{\deg(\Delta_{\mathcal{F}})-1}\pi_{j}(\Tilde{S}_{j})$$
a contradiction.
\end{proof}

In the following we will denote the open sets above by $U(X,H,\mathcal{N})$ e $V(X,H,\mathcal{N})$.

\subsection{Global equations for foliations on \texorpdfstring{$\Projdois$}{P2}}

Let $e\in \mathbb{Z}_{\geq 0}$ and $\corpo$ be an algebraically closed field. A foliation in $\Projdois$ of degree $e$ is determined by a global section of $\Omega_{\Projdois}^{1}(e+2)$. By using the Euler exact sequence  (see \cite[Theorem 8.13]{AGHarthorne}) it follows that a foliation in $\Projdois$ is given, module elements of $\corpo^{*}$, by a $1$-form
$
\Omega = Adx+Bdy+Cdz
$ where  $A,B,C \in \corpo[x,y,z]_{e+1}$ with $\mathcal{Z}(A,B,C)\subset \Projdois$ finite set and such that $i_{R}\Omega = 0$, where $R$ is the radial vector field:$ R = x\partial_{x}+y\partial_{y}+z\partial_{z}. $

Suppose that $\corpo$ has characteristic $p>0$ and let $\mathcal{F}$ be a foliation in $\Projdois$ with normal bundle $N = \mathcal{O}_{\Projdois}(d+2)$ and suppose that $p\nmid \deg(N)$. Suppose that $\mathcal{F}$ is defined by the homogeneous $1$-form:
$
\omega = Adx+Bdy+Cdz
$
and put:
$
d\omega = (d+2)(Ldy\wedge dz-Mdx\wedge dz+Ndx\wedge dy).
$
Let $v \in \mathfrak{X}_{d}(\mathbb{A}_{\corpo}^{3})$ be the homogeneous vector field defined by:
$
v_{\omega} = L\partial_{x}+M\partial_{y}+N\partial_{z}.
$
By the \cite[Proposition 1.1.4]{MR537038} the association $\omega \mapsto v_{\omega}$ defines a bijection between the set of projective $1$-forms of degree $d+2$ and the set of homogeneous vector fields
in $\mathbb{A}_{\corpo}^{3}$ of degree $d$ that has zero divergent
$
div(v_{\omega}) = L_{x}+M_{y}+N_{z}=0.
$
The $p$-divisor is explicitly given by
$
\Delta_{\mathcal{F}} = [i_{v^{p}}\omega] \in \Div(\Projdois). 
$


\begin{ex} Let $\mathcal{F}$ be a foliation of degree two in $\mathbb{P}_{\mathbb{C}}^{2}$ defined by the projective $1$-form
    $$
        \omega = yz^{2}dx-z(4yz+2xz+2y^2)dy+(xyz+4y^2z+2y^3)dz.
    $$
Given a prime number $p\in \mathbb{Z}_{>3}$ consider $\mathcal{F}_{p}$ the foliation obtained by reducing modulo $p$ the $1$-form that defines $\mathcal{F}$. Then, $\mathcal{F}_p$ is not $p$-closed and $\Delta_{\mathcal{F}_p} = 3\{y=0\}+(p+1)\{z=0\}.$    
\end{ex}

\begin{proof} We will show first that  $\mathcal{F}_{p}$ is not $p$-closed for every prime $p>3$. Fix a prime number $p>3$ and consider the foliation $\mathcal{F}_p$ defined in $\Projdois$ by reduction modulo $p$ of the coefficients of the $1$-form $\omega.$ 
Since this problem is local, we can restrict the foliation to the open set
$U = D_{+}(z)\cong \mathbb{A}_{\corpo}^{2}$. In $U$, the foliation is given by the  vector field $v = (4y+2x+2y^{2})\partial_{x}+y\partial_{y}$. Observe that $v(x) = 4y+2x+2y^{2}$. An inductive argument shows that for every $k\geq 3$ we have
$$v^{k}(x) = 2^{2}(2(2^{k-2}+2^{k-3}+\cdots+2+1)+1)y+k2^{k}y^{2}+2^{k}x.$$ 
So, we conclude $v^{p}(x) = 4y+2^{p}x = 4y+2x$, since $0 = 2^{p-1}-1 = (2^{p-2}+\cdots+1)$ in $\mathbb{F}_p$. In particular, we have 
$$
v^{p}(x)v(y)-v^{p}(y)v(x) = y(v^{p}(x)-4y-2y^{2}-2x) = y(-2y^{2}) = -2y^{3} \neq 0.
$$ 
So, the foliation defined by $v$ is not $p$-closed, if $p>3$. This implies that the $p$-divisor of $\mathcal{F}_p$ is given by $\Delta_{\mathcal{F}_p} = 3\{y=0\}+(p+1)\{z=0\}$.
\end{proof}

\subsection{Applications: foliations on \texorpdfstring{$\mathbb{P}_{\mathbb{C}}^{2}$}{P2} without algebraic invariant curves}

The next  two propositions show that, in some cases, the irreducibility of the $p$-divisor for foliations in $\Projdois$ is relate with holomorphic foliations in $\mathbb{P}_{\mathbb{C}}^{2}$ without algebraic invariant curves.

\begin{prop}\label{aplicaI} Let $\mathcal{F}$ be a non-dicritical foliation in $\mathbb{P}_{\mathbb{C}}^{2}$ defined by a projective $1$-form $\Omega = Adx+Bdy+Cdz$. Let $K$ be a number field and suppose that $A,B,C \in \mathcal{O}_{K}[x,y,z]_{d+1}$, where $\mathcal{O}_{K}$ is the integer ring of $K$. Let  $\mathfrak{m} \in \spm(\mathcal{O}_{K})$ of characteristic $p$ and suppose that $p$ does not divides $d+2$. Let $\mathcal{F}_{p}$ be a foliation in $\mathbb{P}_{k(\mathfrak{m})}^{2}$ obtained by reduction modulo  $\mathfrak{m}$ of the coefficients of $\Omega$. If $\Delta_{\mathcal{F}_{p}}$ is irreducible then $\mathcal{F}$ has no algebraic solutions. This can be used to give a simple proof that the Jouanolou foliations of degree two and tree has no algebraic solutions.
\end{prop}

\begin{proof} Suppose, by contradiction, that $\mathcal{F}$ has an algebraic solution $C$. By using Galois automorphism, we can assume that $C$ is defined by an irreducible polynomial over $\mathcal{O}_{K}$. In particular, $C$ is reduced as a curve in $\mathbb{P}_{\mathbb{C}}^{2}$. Let $F \in \mathcal{O}_{K}[x,y,z]$ be the irreducible polynomial defining $C$. By \cite[theorem]{MR1298714} we know that $\deg(F) \leq d+2$. Let $F\otimes k(\mathfrak{m})$ be the polynomial obtained by reduction modulo $\mathfrak{m}$ of $F$. Note that the reduction modulo  $\mathfrak{m}$ preserves the invariance in the sense that the curve describe by $F\otimes k(\mathfrak{m})$ is invariant by the foliation $\mathcal{F}_{p}$. Let $G \in k(\mathfrak{m})[x,y,z]$ be an irreducible factor of $F$. We have that the curve $\{G = 0\} \subset \mathbb{P}_{k(\mathfrak{m})}^{2}$ is $\mathcal{F}_{p}$-invariant and by the Proposition \ref{inv} we conclude that $\{G = 0\} \leq \Delta_{\mathcal{F}_{p}}.$ Since $\Delta_{\mathcal{F}}$ is irreducible we have $\{G=0\} = \Delta_{\mathcal{F}}$. But, this is a contraction by comparison of degrees, since  $p(d-1)+d+2 = \deg(\Delta_{\mathcal{F}_p})>d+2\geq \deg(G)$.
\end{proof}

Let $\mathcal{F}$ be a foliation in $\mathbb{P}_{\mathbb{C}}^{2}$ and $Q$ a reduced singularity of $\mathcal{F}$. Suppose that $Q$ is not degenerated. In this case, we know that if 
$\alpha$ is the eigenvalue of $Q$ then $\alpha \not\in \mathbb{Q}_{+}$. By \cite[Appendice II]{MM2}
we know that there is an analytic coordinate system 
such that the foliation is given by the $1$-form
$
\omega = -\alpha y(1+b(x,y))dx+x(1+a(x,y))dy
$
with $a,b \in \langle x,y \rangle$. In particular,
if $C$ is a reduced algebraic curve that is 
$\mathcal{F}$-invariant with $Q\in C$ then in the analytic coordinate system above the $C$  is given by  $\{x= 0\}$, $\{y = 0\}$ or 
$\{xy = 0\}$. So, computing the Camacho-Sad index we conclude that

\[
   CS(\mathcal{F},C,Q)= \begin{cases} 
            1/\alpha & \text{if}\ C = \{x = 0\},\\ 
            \alpha & \text{if}\ C = \{y = 0\},\\
            \alpha+\alpha^{-1}+2 & \text{if}\ C = \{xy =0\}.
        \end{cases}
\]
In particular, we have that the norm of the index is bounded by a constant which depends only on eigenvalue $\alpha$ of $Q$. More precisely, we have
$|CS(\mathcal{F},C,Q)| \leq |\alpha|+|\alpha|^{-1}+2$. We will use this in the following example.

\begin{prop}\label{aplicaII}  Let $\mathcal{F}$ be a non-degenerated foliation in $\mathbb{P}_{\mathbb{C}}^{2}$ defined by a projective $1$-form $\omega$ with coefficients in $\mathbb{Z}$. Suppose that $\omega$ is primitive, that is, the greatest common divisor of the coefficients of $\omega$ is equal to $1$. Suppose that $\mathcal{F}$ is reduced and define
$$\alpha_{\mathcal{F}}:= \sup\{|\alpha(Q)|\mid Q \in \sing(\mathcal{F})\} \quad \mbox{and}\quad \overline{\alpha}_{\mathcal{F}}:= \sup\{|\alpha(Q)|^{-1}\mid Q \in \sing(\mathcal{F})\}.$$
Let $\beta_{\mathcal{F}} = \alpha_{\mathcal{F}}+\overline{\alpha}_{\mathcal{F}}+2$ and $p$ be a prime number such that $p>(d_{\mathcal{F}}+1)\beta_{\mathcal{F}}^{1/2}$ and suppose that $\Delta_{\mathcal{F}_p}$ is a prime divisor. Then, $\mathcal{F}$ has no algebraic solutions.
\end{prop}

\begin{proof} Suppose by contraction that $\mathcal{F}$ has an algebraic solution $C$. By using Galois automorphism we can assume that $C$ is given by an irreducible polynomial defined over $\mathbb{Z}$. Since $\mathcal{F}$ is not degenerated and reduced, using the Camacho-Sad formula (see \cite[Theorem 3.2]{MR3328860}) we have
$$
d_{C}^{2} = C^{2} \leq \sum_{Q\in \sing(\mathcal{F})\cap C} |CS(\mathcal{F},C;Q)|
\leq \#\sing(\mathcal{F})\beta_{\mathcal{F}} = (d_{\mathcal{F}}^{2}+d_{\mathcal{F}}+1)\beta_{\mathcal{F}}\leq (d_{\mathcal{F}}+1)^{2}\beta_{\mathcal{F}} 
$$
The formula above implies that $d_{C}\leq (d_{\mathcal{F}}+1)\beta_{\mathcal{F}}^{1/2} < p$. So, reducing $C$ modulo $p$ we ensure that the reduction 
$C\otimes \mathbb{F}_{p}$ has no $p$-factor, that  is, there is no an effective divisor, $D$, such that $pD\leq C\otimes \mathbb{F}_p$. Let $E$ be an irreducible factor of $C\otimes \mathbb{F}_p$.  By Proposition \ref{inv} we know that $E$ is a $\mathcal{F}_p$-invariant irreducible curve. Since the $p$-divisor of  $\mathcal{F}_{p}$ is irreducible we conclude that $\Delta_{\mathcal{F}_{p}} = E$. But, this is a contradiction by degree comparison.
\end{proof}

The interesting fact about Proposition \ref{aplicaI} and Proposition \ref{aplicaII} is that we get information about the algebraicity of foliations on $\mathbb{P}_{\mathbb{C}}^{2}$ by using only a maximal ideal.

\subsection{Global equations for foliations on \texorpdfstring{$\Sigma_d$}{Sd}}

In this subsection we recall the construction of the Hirzebruch surfaces and how to represent globally foliations in that surfaces. The reference for this section is \cite{galindo2020foliations}.

Let $\mathbb{G}_{m} = \corpo^{*}$ the multiplicative  group of $\corpo$ and $d \in \mathbb{Z}_{\geq 0}$.  Let  $\mu_{d}$ 
the action of $\mathbb{G}_{m}^{2}$ in $X = (\mathbb{A}_{\corpo}^{2}-0)\times(\mathbb{A}_{\corpo}^{2}-0)$ defined by the morphism:
$$\begin{aligned}
  \mu_{d}: \mathbb{G}_{m}^{2} \times X & \longrightarrow X &  \\
  ((a,b),(x_{0},x_{1};y_{0},y_{1})) &\mapsto (ax_{0},ax_{1},by_{0},\frac{b}{a^{d}}y_{1}).
\end{aligned}$$
The quotient $\Sigma_{d} = X/\mu_{d}$ is a smooth surface defined over $\corpo$ and is it isomorphic to the $d$-Hirzebruch surface, that is, the surface $ \mathbb{P}(\mathcal{O}_{\mathbb{P}_{\corpo}^{1}}\oplus \mathcal{O}_{\mathbb{P}_{\corpo}^{1}}(d))$. It is a ruled surface 
with structural morphism to  $\Projum$ defined by
$$\begin{aligned}
         \pi \colon \Sigma_{d}  & \longrightarrow   \Projum   &  \\
       \overline{(x_{0},x_{1};y_{0},y_{1})}  & \mapsto [x_{0}:x_{1}].
    \end{aligned}$$
Let $d_{1},d_{2} \in \mathbb{Z}$ and $G \in \corpo[x_{0},x_{1},y_{0},y_{1}]$. We say that $G$ is bi-homogeneous of the bi-degree $(d_{1},d_{2})$ if for any monomial
 $x_{0}^{a_{0}}x_{1}^{a_{1}}y_{0}^{b_{0}}y_{1}^{b_{1}}$in the support of $G$ we have $d_{1} = a_{0}+a_{1}-db_{1}$ and $d_{2} = b_{0}+b_{1}$. Let $F$ and $M_{d}$ curves in $\Sigma_{d}$ such that $F$ is a fiber of the structural projection $\pi:\Sigma_{d}\longrightarrow \Projum$ and $M_{d}$ is a section of $\pi$ which satisfies the conditions $M_d\cdot F = 1$ and $M_d^{2}=d$. Then $\{F,M_{d}\}$ forms a base for the vector space $\Num_{\mathbb{Q}}(\Sigma)$. If $D \in \Div(\Sigma_{d})$ is a divisor such that $D\equiv d_{1}F+d_{2}M_d$ then the global sections of the $\mathcal{O}_{\Sigma_{d}}(D)$ correspond to  bi-homogeneous polynomials of  bi-degree $(d_{1},d_{2})$ in $\Sigma_{d}$. In the following, we will use the following description of foliations in Hirzebruch surfaces (see \cite[Proposition 3.2]{galindo2020foliations}).

\begin{prop} Let $d\in \mathbb{Z}_{\geq 0}$, $d_{1},d_{2} \in \mathbb{Z}$ and $N = \mathcal{O}_{\Sigma_{d}}(d_{1}-d+2,d_{2}+2)$. Then, any foliation $\mathcal{F}$ in $\Sigma_{d}$ with normal bundle 
 $N$ is uniquely determined, module $\corpo^{*}$, by a differential $1$-form of the type $\Omega = A_{0}dx_{0}+A_{1}dx_{1}+B_{0}dy_{0}+B_{1}dy_{1}$
where $A_{0},A_{1} \in H^{0}(\Sigma_{d},\mathcal{O}(d_{1}-d+1,d_{2}+2))$, $B_{0} \in H^{0}(\Sigma_{d},\mathcal{O}_{\Sigma_{d}}(d_{1}-d+2,d_{2}+1))$ and $B_{1} \in H^{0}(\Sigma_{d},\mathcal{O}_{\Sigma_{d}}(d_{1}+2,d_{2}+1))$ are bi-homogeneous and satisfies the following conditions: $
x_{0}A_{0}+x_{1}A_{1}-dy_{1}B_{1} = 0$ and $
y_{0}B_{0}+y_{1}B_{1} = 0.$
\end{prop}

\section{The \texorpdfstring{$p$}{p}-divisor for foliations on \texorpdfstring{$\Projdois$}{P2}}\label{pdivisoremP2}

In this section, we investigate the structure of the $p$-divisor for generic foliations in the projective plane. In the next section, $\corpo$ will be denote an algebraically closed field that has characteristic $p>0$. The following lemma will be important to the next sections.

\begin{lemma}\label{pullback} Let $F \in \corpo[x_{0},\ldots,x_{n}]$ be a reduced polynomial and $l \in \mathbb{Z}$ positive integer with $\mdc(l,p) = 1$. Suppose that $x_{0}\nmid F$. Then, $F(x_{0}^{l},x_{1},\ldots,x_{n})$ is reduced.
\end{lemma}

\begin{proof} First, note that if $x_0$ does not occurs in $F$ then is nothing to prove. So, we can suppose that $x_0$ occurs in $F$. We first consider the case where $F$ is irreducible. Note that it is sufficient to consider the case where $n=0$. 
Indeed, if $R = \corpo[x_{1},\ldots,x_{n}]$ and $K$ is their fraction field then by the Gauss Lemma (see \cite[Theorem 2.1]{MR1878556}) we know that $F \in R[x_{0}]$ is irreducible if and only if $F = a_{0}+a_{1}x_{0}+\cdots+a_{d}x_{0}^{d} \in K[x_{0}]$ is irreducible and has content $1$, that is, the greatest common divisor of the coefficients is equal to $1$. Since we are assuming that $F$ is irreducible we have, in particular, that $F$ is irreducible over $K[x_{0}]$ where $K = \corpo(x_{1},\ldots,x_{n})$. Let $g(x_{0})= F(x_{0}^{l}) \in K[x_{0}]$. Note that the content of $g$ is equal to $1$, since the coefficients of $g$ are equal to the coefficients of $F$. In the other hand, the irreducibility of $F$ implies that $g$ is reduced: indeed, to see this note that it is sufficient to proof that $\mdc(g(x_{0}),\frac{dg}{dx_{0}}) = 1$. Now, by taking derivatives we have
$
\frac{dg}{dx_{0}} = lF^{'}(x_{0}^{l})x_{0}^{l-1}\neq 0.
$
Since $F$ is irreducible we ensure that $F^{'}(x_{0}^{l})$ and $F(x_{0}^{l})$ are coprime. In particular, we have $\mdc(g(x_{0}),g^{'}(x_{0})) = 1$. So,  $g(x_{0})$ is reduced in $R[x_{0}] = \corpo[x_{0},x_{1},\ldots,x_{n}]$.

Now we will consider the general case, where $F$ is reduced. Let $R = \corpo[x_1,\ldots,x_n]$ and $K$ their fraction field. Note that, without loss of generality, we can assume that 
$F\in R[x_0]$ has content equal to $1$. Let $G$ and $H$ two irreducible factors of $F$. We will show that $\tilde{G} = G(x_{0}^{l},x_{1}\ldots,x_{n})$ and $\tilde{H} = H(x_{0}^{l},x_{1}\ldots,x_{n})$ do not have common irreducible factors. By the Gauss Lemma (see \cite[Theorem 2.1]{MR1878556}) is it sufficient to show that $\tilde{G}$ and $\tilde{H}$
do not have common irreducible factors over 
$K[x_{0}]$, where $K = \corpo(x_{1},\ldots,x_{n})$. In the other hand, since  $G$ and $H$ has not common factors over $\corpo[x_{0},\ldots,x_{n}]$ we have, in particular, 
that there is no common irreducible factors over $K[x_{0}]$. So, there exists $A(x_{0}), B(x_{0}) \in K[x_{0}]$ such that
$
A(x_{0})G+B(x_{0})H = 1.
$
Specializing the identity above to $x_{0}^{l}$ we conclude that
$
A(x_{0}^{l})\tilde{G}+B(x_{0}^{l})\tilde{H} = 1
$
and so $\tilde{G}$,$\tilde{H}$ have no common irreducible factors over $K[x_{0}]$. So, we consider only the case where the polynomial $F$ is irreducible by considering their decomposition in irreducible factors.
\end{proof}

\begin{OBS} Let $l\in \mathbb{Z}_{>0}$ with $\mdc(l,p) =1$. In general, we do not ensure that if  $f \in \corpo[x]$ is irreducible then $g(x) = f(x^{l})$ is irreducible. Indeed, let $p>2$ and consider $f(x) = x-1$. If $l \in 2\mathbb{Z}$ then we have
$
g(x) =f(x^{l}) = (x^{\frac{l}{2}}-1)(x^{\frac{l}{2}}+1).
$
\end{OBS}

In the following, we present the proof of the Theorem \ref{casoplano}. Recall the statement.

\begin{thmA} Let $\corpo$ be an algebraically closed field of characteristic $p>0$. A generic foliation in $\Projdois$ of degree $d\geq 1$ with $p\nmid d-1$ has reduced $p$-divisor. More precisely, if $N$ is a fixed bundle then \footnote{Section \ref{notacoes} for notations}$U(\mathbb{P}_{\corpo}^{2},\mathcal{O}_{\mathbb{P}_{\corpo}^{2}}(1),N) \neq \emptyset$, if $\deg(N) = 3$ or if $\deg(N)>3$ with $p\nmid \deg(N)-3$.
\end{thmA}

We will consider first the case where $d\in \{1,2\}$ and after that we will use the case $d = 2$ to show the general case. We will divide the proof in propositions.

\begin{prop} A generic foliation in $\Projdois$ of degree $d\in \{1,2\}$ has reduced $p$-divisor. More precisely, we have $U(\mathbb{P}_{\corpo}^{2},\mathcal{O}_{\mathbb{P}_{\corpo}^{2}}(1),N) \neq \emptyset$, if $\deg(N) =\{3,4\}$.
\end{prop}

\begin{proof} If $\deg(N) = 3$, that is, $d =1$ then there are examples of foliations with reduced $p$-divisor. Indeed, if $\alpha \in \corpo-\mathbb{F}_{p}$ consider the foliation 
$\mathcal{F}_{\alpha}$ in $\Projdois$ given by the $1$-form $\Omega_{\alpha} = -yzdx+\alpha xzdy+(1-\alpha)xydz$. The condition $\alpha \not\in \mathbb{F}_{p}$ implies that $\mathcal{F}_{\alpha}$ is not $p$-closed (see Lemma \ref{intlocal}) with $\Delta_{\mathcal{F}_{\alpha}} = \{x=0\}+\{y=0\}+\{z=0\}$. Now, consider the case where $\deg(N) = 4$ and let $D_{+}(z) = \{[x:y:z]\in \Projdois\mid z\neq 0\}$. We will get the example by compactification via the isomorphism
$$\begin{aligned}
          \Phi\colon D_{+}(z)& \longrightarrow   \mathbb{A}_{\corpo}^{2} &  \\
        [x:y:z]   & \mapsto \left(\frac{x}{z},\frac{y}{z}\right)
    \end{aligned}$$
    
\noindent of a foliation $\mathcal{G}$ in $\mathbb{A}_{\corpo}^{2}$ given by the $1$-form: $\omega = ydx-xdy+\omega_{2}$ where $\omega_{2} = a(x,y)dx+b(x,y)dy$ for $a,b \in \corpo[x,y]$ generic homogeneous polynomial of degree two. Note that $\mathcal{G}$ has tree $\mathcal{G}$-invariant lines which contains the point $(0,0)$. Indeed, the lines are given explicitly by the polynomial  $i_{R}\omega_{2} = l_{1}l_{2}l_{3}$, where $R$ is the radial vector field in $\mathbb{A}_{\corpo}^{2}$: $R = x\partial_{x}+y\partial_{y}$.
\noindent Since $a$ and $b$ are generic, we can assume that $l_{1}, l_{2}$ and $l_{3}$ has multiplicity $1$ along $\Delta_{\mathcal{G}}$. Indeed, we choose the lines $l_{1},l_{2},l_{3}$ such that for each 
$i$ we ensure that $l_{i}\cap l_{\infty} = \{P_{i}\}$  is a $p$-reduced  singularity of $\mathcal{G}$. In this case, by 
\cite[Fact 2.8]{MR3687427} we ensure that $l_{i}$ occurs with multiplicity $1$ in the $p$-divisor $\Delta_{\mathcal{G}}$. Let 
$\mathcal{F}$ be a foliation obtained via compactification of $\mathcal{G}$ in $\Projdois$, via $\Phi$. By Lemma \ref{intlocal} we know that $\mathcal{F}$ is not $p$-closed with four invariant lines, $l_{1},l_{2},l_{3}$ and $l_{\infty} = \{z= 0\}$, and so $\Delta_{\mathcal{F}} = l_{1}+l_{2}+l_{3}+l_{\infty}+C
$ for some curve $C$ of degree $p$. 

We will show that a generic choice of $a,b$ implies that  $C$ is irreducible and has $Q = [0:0:1]$ as a singularity of multiplicity $m_{Q}(C) = p-1$. For this, let $\pi: Bl_{Q}(\Projdois)\longrightarrow \Projdois$ the blowup at $Q$ fix $E$ the exceptional divisor and  
$F$ a fiber of the natural projection $\pi: Bl_{Q}(\Projdois)\longrightarrow \Projum$. Note that 
$\{F,E\}$ forms a base to the vector space $\Num_{\mathbb{Q}}(Bl_{Q}(\Projdois))$  and satisfies the following conditions $E^{2} = -1$, $E\cdot F = 1$ and $F^{2} = 0$. Let $\mathcal{H} = \pi^{*}\mathcal{F}$ the induced foliation. Since $Q$ is a radial singularity we have (see \cite[Chapter 2, Section 3]{MR3328860})
$N_{\mathcal{H}}^{*} = \pi^{*}N_{\mathcal{F}}^{*}+2E$ and $K_{\mathcal{H}} = \pi^{*}K_{\mathcal{F}}-E$. Let $H$ be a line containing the point $Q$. Since $\pi^{*}N_{\mathcal{F}}^{*} = \pi^{*}(-4H) \equiv -4\pi^{*}H \equiv -4(F+E)$ and $\pi^{*}K_{\mathcal{F}} \equiv  \pi^{*}H \equiv F+E$, we conclude 
\[
\Delta_{\mathcal{H}} \equiv pK_{\mathcal{H}}+N_{\mathcal{H}} \equiv 2E+\deg(\Delta_{\mathcal{F}})F = (p+4)F+2E.
\]
Let $\Tilde{l}_{1}$, $\Tilde{l}_{2},\Tilde{l}_{3}$ and $\Tilde{l}_{\infty}$ be the strict transform of $l_{1},l_{2}, l_{3}$ and $l_{\infty}$ respectively. Since $\pi\colon Bl_{Q}(\Projdois)-E \longrightarrow \mathbb{P}_{\corpo}^{2}-\{Q\}$ is an isomorphism we verify that $\Tilde{l}_{1}$, $\Tilde{l}_{2},\Tilde{l}_{3}$ and $\Tilde{l}_{\infty}$ occurs in $\Delta_{\mathcal{H}}$, that is $
\Delta_{\mathcal{H}}-\Tilde{l}_{1}-\Tilde{l}_{2}-\Tilde{l}_{3}-\Tilde{l}_{\infty} \geq 0$. Moreover, since $a,b$ are generic we can assume that $\od_{l_{i}}(\Delta_{\mathcal{H}})\geq 0$ for $i\in\{1,2,3,\infty\}$.
As $
(\Delta_{\mathcal{H}}-\Tilde{l}_{1}-\Tilde{l}_{2}-\Tilde{l}_{3}-\Tilde{l}_{\infty})\cdot F = 1$
there exist an irreducible curve $C \subset Bl_{Q}(\mathbb{P}^{2})$ that is $\mathcal{H}$-invariant and such that $C\cdot F = 1$. In particular, we have $C \equiv E+\alpha F$. Note that $\alpha>0$, since $Q$ is a radial singularity and, in particular, $E$ is not $\mathcal{H}$-invariant. Writing $
\Delta_{\mathcal{H}}-\Tilde{l}_{1}-\Tilde{l}_{2}-\Tilde{l}_{3}-\Tilde{l}_{\infty} = C+R$ for some divisor $R$ we conclude (using $C\cdot F =1$) that $R \equiv bF$, for some $b \in \mathbb{Z}_{\geq 0}$. Now, since $\omega_{2}$ is generic we may assume $b = 0$. Indeed, if $b\neq 0$, we see that
there exist a fiber of the natural projection $p\colon Bl_{Q}(\Projdois)\longrightarrow \Projum$ that is  $\mathcal{H}$-invariant (see Proposition \ref{inv}). By projecting this fiber in  $\Projdois$, we obtain a line which contain $Q$ that is $\mathcal{F}$-invariant. Since $l_{1}$,$l_{2}$,$l_{3}$ and $l_{\infty}$ occurs in $\Delta_{\mathcal{F}}$ with multiplicity $1$ and since they are the unique invariant lines of $\mathcal{F}$, we obtain a contraction. So, we conclude that $R = 0$ and $C$ is an irreducible  curve that is $\mathcal{H}$-invariant with $C\equiv E+pF$. Projecting $C$ via the map $\pi$ we obtain in $\Projdois$ an irreducible algebraic curve of degree $p$ that has $Q$ as a singularity and with multiplicity $p-1$. indeed, $\pi_{*}C$ is irreducible and the degree is given by
$$\deg(\pi_{*}(C)) = \pi_{*}C\cdot H = \pi^{*}\pi_{*}C\cdot \pi^{*}H = (E+pF)\cdot(E+F) = -1+p+1 = p.$$ 
In other hand, the multiplicity can be computed in the following way:
$$
\pi^{*}\pi_{*}C = \Tilde{C}+m_{Q}(C)E \equiv (E+pF)+m_{Q}(C)E \Longrightarrow  0 = \pi^{*}\pi_{*}C\cdot E = (E+pF)\cdot E-m_{Q}(C)
$$
and so $
0 = -1+p-m_{Q}(C)$ which implies $m_{Q}(C) = p-1$. This concludes the proof for $d \in \{1,2\}$.
\end{proof}

\begin{prop} A generic foliation in $\Projdois$ with normal bundle $N$ and of degree $d\geq 3$  with $p\nmid d-1$ has reduced $p$-divisor. More precisely, $U(\mathbb{P}_{\corpo}^{2},\mathcal{O}_{\mathbb{P}_{\corpo}^{2}}(1),N) \neq \emptyset$, if $\deg(N)\geq 5$ with $p\nmid \deg(N)-3$.
\end{prop}
\begin{proof} We use the case  $d = 2$ to show the general case. Let $e \in \mathbb{Z}_{>0}$ such that $\mdc(e,p) = 1$ and $\mathcal{F}$ be a foliation in $\Projdois$ of degree $2$ that has reduced $p$-divisor. By the precedent proposition, we can assume that there exists a foliations that leaves invariant an irreducible algebraic curve $C$ of degree $p$ and four lines: $l_1,l_2,l_3$ and $l_{\infty}$ where $l_{1}\cap l_{2} = l_{0}\cap l_{1} = l_{0}\cap l_{2} = \{Q\}$. For simplicity, we suppose that  $l_1 = \{x=0\}$, $l_2 = \{y=0\}$ and $l_3 = \{ux+vy=0\}$ for some constants $u,v \in \corpo^{*}$. In this case, $\mathcal{F}$ is defined by a $1$-form of the type
$
\Omega_{1} = yzA_0dx+xzA_1dy+xyA_2dz 
$
for certain $A_0,A_1,A_2 \in \corpo[x,y,z]$ homogeneous of degree $1$ and such that $A_0+A_1+A_2 = 0$. Consider the finite morphism
$$\begin{aligned}
         \Phi_e\colon \Projdois & \longrightarrow   \Projdois  &  \\
        [x_{0}:x_{1}:x_{2}]  & \mapsto [x_{0}^{e}:x_{1}^{e}:x_{2}^{e}].
    \end{aligned}$$
Let $\mathcal{H}$ the foliation in $\Projdois$ defined by the saturation of the $1$-form
$$
\Omega = \Phi_{e}^{*}\Omega_{1} =  e(xyz)^{e-1}[yzA_0(x^{e},y^{e},z^{e})dx+xzA_1(x^{e},y^{e},z^{e})dy+xyA_2(x^{e},y^{e},z^{e})dz].
$$
Observe that $\mathcal{H}$ is a foliation of degree $e+1$ which is not $p$-closed. Since $\mdc(e,p)=1$, we can use the Lemma \ref{pullback} 
to ensure that $\Phi_{e}^{*}l_3$ and $\Phi_{e}^{*}C$ are reduced curves with irreducible distinct components. In particular, by the Proposition \ref{inv} it follows that $\Phi_{e}^{*}l_3$ and $\Phi_{e}^{*}C$ are $\mathcal{H}$-invariant curves. We claim that 
 $\mathcal{H}$ has $p$-divisor
$
\Delta_{\mathcal{H}} = \{x=0\}+\{y=0\}+\{z=0\}+\Phi_{e}^{*}l_3+\Phi_{e}^{*}C.
$
Indeed, note that $\Delta_{\mathcal{H}}$ is a divisor with degree given by the formula 
$$\deg(\Delta_{\mathcal{H}}) = pK_{\mathcal{H}}\cdot \mathcal{O}_{\mathbb{P}_{\corpo}^{2}}(1) +N_{\mathcal{H}}\cdot \mathcal{O}_{\mathbb{P}_{\corpo}^{2}}(1) = pe+e+3.$$ 
In the other hand, by construction we know that the curves $\{x=0\}$,$\{y=0\}$,$\{z=0\}$, $\Phi_{e}^{*}l_3$ and $\Phi_{e}^{*}C$ are $\mathcal{H}$-invariant. In particular,
$
\Delta_{\mathcal{H}} \geq \{x=0\}+\{y=0\}+\{z=0\}+\Phi_{e}^{*}l_3+\Phi_{e}^{*}C .
$
By comparison of the degrees we conclude the equality
$\Delta_{\mathcal{H}}= \{x=0\}+\{y=0\}+\{z=0\}+\Phi_{e}^{*}l_3+\Phi_{e}^{*}C.$
So, it follows that $\mathcal{H}$ is a foliation of degree $e+1$ in $\Projdois$ with reduced $p$-divisor. This finishes the proof of the proposition.
\end{proof}

\begin{OBS} \label{proj} In the  proof of the Theorem \ref{casoplano} we saw that is possible to find a foliation $\mathcal{F}$ of degree $2$ in $\Projdois$ that is not $p$-closed with $p$-divisor in the form
$
\Delta_{\mathcal{F}} = l_{1}+l_{2}+l_{3}+l_{4}+C
$
where $C$ is an irreducible algebraic curve of degree $p$ and $l_{i}\neq l_{j}$ if $i\neq j$. Let $P,Q \in \Projdois$ and fix $\Phi:\Projdois\longrightarrow \Projdois$ an automorphism
such that $\Phi(P), \Phi(Q) \ \not\in C$. Then, $\Phi^{*}\mathcal{F}$ is foliation that is not $p$-closed with $p$-divisor 
$
\Delta_{\Phi^{*}\mathcal{F}} = \Phi^{*}l_{1}+\Phi^{*}l_{2}+\Phi^{*}l_{3}+\Phi^{*}l_{4}+\Phi^{*}C
$
where $\Phi^{*}C$ is an irreducible algebraic curve of degree $p$ such that $P,Q \not\in \Phi^{*}C$.
\end{OBS}

\section{The \texorpdfstring{$p$}{p}-divisor for foliations on \texorpdfstring{$\Projumprod$}{P1P1}}\label{pdivisorP1P1}

In the following, we will fix $x_{0},x_{1}$ (resp. $y_{0},y_{1}$) as the coordinates functions of the first factor (resp. second factor) of $\Projumprod$. Let $\pi_{1}$ the projection of $\Projumprod$ over the first factor and $\pi_{2}$ the projection of  $\Projumprod$ over the second factor. Let $F$ and $M$ fibers of $\pi_{1}$ and $\pi_{2}$ respectively. Recall that $\{F,M\}$ is a base for the vector space $\Num_{\mathbb{Q}}(\Projumprod)$ which satisfies the following numerical conditions $F^{2} = 0$,  $F\cdot M =1$, and $M^{2} = 0.$ Observe that for any fibers $F_{0} = \pi_{1}^{-1}(\pt)$ and $M_{0} = \pi_{2}^{-1}(\pt)$ we have $F_{0}\equiv F$ and $M_{0} \equiv M$. 

\begin{lemma}\label{compactificacao} Let $C = \mathcal{Z}(f) \subset \mathbb{A}_{\corpo}^{2}$ be an irreducible algebraic curve of degree $d$ and $f = f_{k}+f_{k+1}+\cdots+f_{d}$ the decomposition of $f$
in homogeneous terms as element of $\corpo[x,y]$. Suppose that 
$x^{d}$ and $y^{d}$ effectively occurs in $f_{d}$ and
consider $\overline{C}$ the bi-projectivization of $C$ in $\Projumprod$ via the isomorphism:
$$\begin{aligned}
          \Phi\colon U_{11} & \longrightarrow \mathbb{A}_{\corpo}^{2} &  \\
        ([x_{0}:x_{1}],[y_{0}:y_{1}])  & \mapsto \left(\frac{x_{0}}{x_{1}},\frac{y_{0}}{y_{1}}\right).
    \end{aligned}$$
Then, $\overline{C}$ is irreducible and if $\{F,M\}$ is a base for $\Num_{\mathbb{Q}}(\Projumprod)$, with $F^{2} = 0$, $F\cdot M = 1$ and $M^{2} = 0$ then
$
\overline{C} \equiv \deg(C)(F+M).
$
\end{lemma}

\begin{proof} The compactification $\overline{C}$ is defined by the following polynomial
$$
F = (x_{1}y_{1})^{d-k}f_{k}(x_{0}y_{1},x_{1}y_{0})+(x_{1}y_{1})^{d-k-1}f_{k+1}(x_{0}y_{1},x_{1}y_{0})+\cdots+f_{d}(x_{0}y_{1},x_{1}y_{0})
$$
It is easy to see that $F$ is a homogeneous polynomial of  bi-degree $(d,d)$, so that we need only to check their irreducibility.  In the  other hand, note that if exists an irreducible factor $H$ of $F$ then necessarily we have $H \not\in \langle x_{1}y_{1}\rangle$ since the hypothesis in
$f_{d}$ implies  that $f_{d}(x_{0}y_{1},x_{1}y_{0})\not\in \langle x_{1}y_{1}\rangle$. But, specializing to $x_{1} = 1, y_{1} = 1$ we conclude that $f$ is reducible. Contradiction!  
\end{proof}

\begin{ex} Let $C$ be an affine curve in $\mathbb{A}_{\corpo}^{2}$ defined by the equation
$f = x+y+xy$. Using the map
$$\begin{aligned}
          \Phi\colon U_{11} & \longrightarrow  \mathbb{A}_{\corpo}^{2} &  \\
        ([x_{0}:x_{1}],[y_{0}:y_{1}])  & \mapsto \left(\frac{x_{0}}{x_{1}},\frac{y_{0}}{y_{1}}\right)
    \end{aligned}$$
and projectivizing $C$ in $\Projumprod$ via $\Phi$ we obtain the curve $\overline{C} = \mathcal{Z}(x_{0}y_{1}+x_{1}y_{0}+x_{0}y_{0})$ of bi-degree $(1,1)\neq (2,2)$. This shows that the condition in Lemma \ref{compactificacao} is necessary.
\end{ex}

\begin{thm} \label{casop1p1} Let $d_{1},d_{2}\in \mathbb{Z}_{\geq0}$ such that $p\nmid d_{i}$, if $d_{i}\neq 0$. Then, a generic foliation in $\Projumprod$ with normal bundle $N\equiv (d_{1}+2)F+(d_{2}+2)M$ has reduced $p$-divisor. More precisely, $U(\Projumprod,\mathcal{O}_{\Projumprod}(F+M),N)\neq\emptyset$ if
\begin{itemize}
    \item $N\cdot F-2\geq 0$ and $p\nmid N\cdot F-2$ (if nonzero);
    \item  $\deg(N)-N\cdot F+d-2 \geq 0$ and $p\nmid \deg(N)-N\cdot F-2$ (if nonzero).
\end{itemize}
\end{thm}
The proof of the Theorem \ref{casop1p1} is divided in propositions. We show that given $d_{1},d_{2}\in \mathbb{Z}_{\geq 0}$ such that $p\nmid d_{i}$ (if  $d_{i}\neq 0$) we can find a foliation $\mathcal{G}$ in $\Sigma_{0} = \Projumprod$ with the following properties:
\begin{enumerate}[(i)]
    \item \label{um}$K_{\mathcal{G}}\equiv d_{1}F_{0}+d_{2}M_{0}$ where $F_{0} = \{x_{0}= 0\}$ and $M_{0} = \{y_{0}=0\}$.
    \item \label{dois} $F_{0}$ and $M= \{y_{1}=0\}$ are $\mathcal{G}$-invariant curves with $\{Q\} = F_{0}\cap M$ a $p$-reduced singularity of $\mathcal{G}$.
    \item \label{tres} The $p$-divisor $\Delta_{\mathcal{G}}$ is reduced.
\end{enumerate}
We will consider two cases: A and B.

\subsection{Case A}

In this subsection, we proof, in particular, that a generic foliation  $\mathcal{G}$ in $\Projumprod$ with cotangent divisor $K_{\mathcal{G}} \equiv d_1F_0+d_2M_0$ has reduced $p$-divisor when 
$$(d_1,d_2) \in \{(0,0)\}\cup\{(l,0)\in \mathbb{Z}^2\mid l>0 \mbox{ and } p\nmid l \}\cup\{(0,l)\in \mathbb{Z}^2\mid l>0 \mbox{ and } p\nmid l \}.$$ We start with the case $d_1 = d_2 = 0$.
\begin{prop} A generic foliation in $\Projumprod$ with cotangent divisor trivial $K_{\mathcal{G}}\equiv 0$ satisfies (\ref{um}),(\ref{dois}) e (\ref{tres}).
\end{prop}

\begin{proof}  Let $\omega$  be a $1$-form in $\mathbb{A}_{\corpo}^{2}$ given by 
$
\omega= \alpha ydx+xdy \quad\mbox{with}\quad \alpha \not\in \mathbb{F}_{p}.
$
Consider $\mathbb{A}_{\corpo}^{2}$ as an open subset of  $\Projumprod$ via the  isomorphism
$$
\Phi\colon U_{11} \longrightarrow \mathbb{A}_{\corpo}^{2} \qquad ([x_{0},x_{1}],[y_{0},y_{1}]) \mapsto \left(\frac{x_{0}}{x_{1}},\frac{y_{0}}{y_{1}}\right)
$$
where $U_{11} = \{y_{1}\neq 0\}\cap\{x_{1}\neq 0\}\subset \Projumprod$. Projectivizing the $1$-form $\omega$ in $\Projumprod$ by the map $\Phi$ we obtain:
$$
\Phi^{*}\omega = \frac{\alpha y_{0}x_{1}dx_{0}-\alpha y_{0}x_{0}dx_{1}}{x_{1}^{2}y_{1}}+\frac{x_{0}y_{1}dy_{0}-y_{0}x_{0}dy_{1}}{x_{1}y_{1}^{2}}
$$
and considering their saturation it follows that
$
\Omega = \alpha x_{1}y_{0}y_{1}dx_{0}-\alpha x_{0}y_{0}y_{1}dx_{1}+x_{0}x_{1}y_{1}dy_{0}-x_{0}x_{1}y_{0}dy_{1}
$
a projective $1$-form that is bi-homogeneous which defines a foliation $\mathcal{F}_{\Omega}$ with $K_{\mathcal{F}_{\Omega}} \equiv 0$. In particular, satisfies (\ref{um}). Note that $F_0$, $F_{1} = \{x_{1} =  0\}$, $M = \{y_{1} = 0\}$ and $M_{0}$ are $\mathcal{F}_{\Omega}$-invariant curves, so that
$
\Delta_{\mathcal{F}_{\Omega}} \geq M_{0}+M+F_{1}+F_{2}.
$
In the other hand,
$
\Delta_{\mathcal{F}_{\Omega}} \equiv pK_{\mathcal{F}}+N_{\mathcal{F}} \equiv N_{\mathcal{F}} \equiv 2F_{0}+2M_{0}
$
and by degree comparison we conclude the equality 
$
\Delta_{\mathcal{F}_{\Omega}} = M_{0}+M+F_{1}+F_{2}.
$
In particular, $\mathcal{F}_{\Omega}$ satisfies (\ref{tres}). Now, let $U_{10}$ be the open set given by $\{y_{0}\neq 0\}\cap \{x_{1} \neq 0\}$.  Observe that there exists an isomorphism
$$\begin{aligned}
          \psi \colon U_{10} & \longrightarrow \mathbb{A}_{\corpo}^{2} &  \\
        \left([x_{0}:x_{1}],[y_{0}:y_{1}]\right)  & \mapsto \left(\frac{x_{0}}{x_1},\frac{y_1}{y_0}\right) = (x,y).
    \end{aligned}$$
Restricting the foliation $\mathcal{F}_{\Omega}$ to the open $U_{10}\cong \mathbb{A}_{\corpo}^{2}$ we obtain a foliation given by the $1$-form:
$
\sigma = \alpha ydx-xdy
$
and so $M\cap F_{0} = \{y=0\}\cap\{x=0\} = \{(0,0)\}$ is a $p$-reduced singularity. In this way, we conclude that $\mathcal{F}_{\Omega}$ satisfies (\ref{dois}). This finishes the proof for $d_2 = d_1 = 0$.
\end{proof}

We pass now to the a more general case.

\begin{prop} Let $d_{1},d_{2}\in \mathbb{Z}_{\geq0}$ such that $p\nmid d_{i}$, if $d_{i}\neq 0$. Suppose that
$$(d_1,d_2) \in \{(0,0)\}\cup\{(l,0)\in \mathbb{Z}^2\mid l>0 \mbox{ and } p\nmid l \}\cup\{(0,l)\in \mathbb{Z}^2\mid l>0 \mbox{ and } p\nmid l \}. 
$$
Then, a generic foliation in $\Projumprod$ with cotangent divisor $K_{\mathcal{G}}\equiv d_{1}F+d_{2}M$ satisfies (\ref{um}),(\ref{dois}) and (\ref{tres}).
\end{prop}

\begin{proof} The case $d_1= d_2 = 0$ was considered in the precedent proposition. We will consider the case where $d_{1} = l$ and $d_{2} = 0$ for $l>0$ with $p\nmid l$. By the symmetry of the the problem we will automatically consider the case where $d_{1} = 0$ and $d_{2} = l$ for $l>0$ with $p\nmid l$.

\medskip

\noindent \textbf{Case $l=1$:} Let $\mathcal{F}$ be a Riccati foliation with respect the first projection $\pi_1\colon\Projumprod\longrightarrow \Projum.$ Recall (see \cite[Chapter 4, Section 1]{MR3328860}) that a Riccati foliation is defined as a foliation in $\Projumprod$ whose the general fiber $F$ of $\pi_1$ is transverse to $\mathcal{F}$. Suppose that
        \begin{enumerate}[(a)]
            \item \label{A} The foliation $\mathcal{F}$ leaves invariant only tree fibers of the first projection $\pi_1$.  Denote that fibers by $F_{0},F_{1},F_{2}.$
            \item \label{B} The foliation $\mathcal{F}$ leaves invariant only one fiber of the second projection $\pi_{2}$. Denote that fiber by $M$.
            \item \label{C} The intersections $\{Q_{i}\} = F_{i}\cap M$ are $p$-reduced singularities of $\mathcal{F}$.
    \end{enumerate}
    We show that that foliation satisfies 
 (\ref{um}),(\ref{dois}) and (\ref{tres}). Note that by (\ref{A}) (see \cite[Section 1, Chapter 4]{MR3328860})
        $$
            K_{\mathcal{F}} = \pi_{1}^{*}K_{\Projum}+F_{0}+F_{1}+F_{2} \equiv -2F+3F = F.
        $$
    So, $\mathcal{F}$ satisfies (\ref{um}). Note that the item (\ref{C}) ensures that $\mathcal{F}$ is not $p$-closed (see Lemma \ref{intlocal}) and that $\mathcal{F}$ satisfies (\ref{dois}), module a change of coordinates. We will show that 
 $\mathcal{F}$ satisfies (\ref{tres}). Observe that
    $$
        D = \Delta_{\mathcal{F}}-F_{0}-F_{1}-F_{3}-M \geq 0
    $$
    and by (\ref{B}) it follows that  $\od_{M}(\Delta_{\mathcal{F}}) = \od_{F_{i}}(\Delta_{\mathcal{F}}) = 1 $ for all $i$ (see \cite[Fact 2.8]{MR3687427}). Since $\Delta_{\mathcal{F}} = pK_{\mathcal{F}}+N_{\mathcal{F}} \equiv pF_0+3F_0+2M_0 \equiv (p+2)F_0+2M_0$, the following equality holds
    $$
        D\cdot F = \Delta_{\mathcal{F}}\cdot F -M\cdot F =2-1 = 1.
    $$
    So, there exists an irreducible curve 
    $C\leq D$ such that $C\cdot F = 1$. Write
    $D = C+R$ and note that we have $C \equiv uF_{0}+M_{0}$ and $R \equiv vF_{0}$, for some $u,v \in \mathbb{Z}_{\geq 0}$. We will show that $v \equiv 0 \mod p$. Suppose, by contradiction, that $v \neq 0$. Then, since
    $R\equiv vF_{0}$ we conclude by the Proposition \ref{inv} that there exists a fiber of the first projection, $F_{4}\leq R$, that is $\mathcal{F}$-invariant. In the other hand, by construction of $\mathcal{F}$, we know that $F_{0}$,$F_{1}$ and $F_{2}$ are the complete list of $\mathcal{F}$-invariant fibers. Moreover, all that fibers has multiplicity $1$ along the $p$-divisor
     $\Delta_{\mathcal{F}}$. So, the fiber 
 $F_{4}$ can not exist and we conclude that $v \equiv 0 \mod p$. In the other hand,
        $$
            u+v = C\cdot M_{0}+R\cdot M_{0} =  D\cdot M_{0} = \Delta_{\mathcal{F}}\cdot M_{0}-3 = p+3-3 = p.
        $$
    So,  if $v \neq 0$ the condition $v \equiv 0 \mod p$ implies that $v = p$ and by consequence it follows $u = 0$. But this implies $C \equiv M$ so that $C$ is a fiber of the second projection which is  $\mathcal{F}$-invariant. This is a contradiction, since by construction we know that $M$ is the unique fiber of the second projection that is $\mathcal{F}$-invariant and satisfies $\od_{M}(\Delta_{\mathcal{F}}) = 1$. So, $v = 0$ and $C$ is an irreducible curve $\mathcal{F}$-invariant with $C \equiv pF_0+M_0$ and $\mathcal{F}$ has $p$-divisor given by
        $$
            \Delta_{\mathcal{F}} = F_0+F_1+F_2+M+C
        $$
    which is reduced. So, $\mathcal{F}$ satisfies (\ref{um}),(\ref{dois}) and (\ref{tres}) and we conclude the argument to the case $(d_1,d_2) =(1,0)$.

    \medskip 

  \noindent \textbf{Case $l>0$ and $p\nmid l$:} We use the precedent case ($l=1$) to study the present case.  Let $\mathcal{F}$ be a foliation in $\Projumprod$ as in the precedent case. As we had proved, a generic foliation which satisfies (\ref{A}),(\ref{B}) and (\ref{C}) has $p$-divisor in the form
        $$
            \Delta_{\mathcal{F}} = F_0+F_1+F_3+M+C
        $$
    where $C$ is an irreducible curve with $C\equiv pF_0+M_0$, $F_0,F_1$ and $F_2$ are fibers of the first projection $\pi_{1}$ and $M$ is a fiber of the second projection $\pi_2$.
    
    Let $\Phi\colon \Projumprod \longrightarrow\Projumprod $ be a finite ramified map of degree $l$ with ramification divisor $R$. Suppose that $\Phi$ ramifies only along the curves $F_0$ and $F_1$. Let $\mathcal{G} = \Phi^{*}\mathcal{F}$ the foliation defined by the pull-back of $\mathcal{F}$ by $\Phi$. Explicitly, if $\Omega$ is the bi-homogeneous $1$-form which defines  $\mathcal{F}$ then $\mathcal{G}$ is the foliation defined by the saturation of $\Phi^{*}\Omega$. Observe that since $F_0$ and $F_1$ are $\mathcal{F}$-invariant $K_{\mathcal{G}} = \Phi^{*}K_{\mathcal{F}}$. So, $ \Phi^{*}K_{\mathcal{F}} = \Phi^{*}F = lF$ and we conclude that $\mathcal{G}$ satisfies (\ref{um}) and $N_{\mathcal{G}} \equiv (l+2)F+2M$ and $\Delta_{\mathcal{G}}\equiv  pK_{\mathcal{G}}+N_{\mathcal{G}}\equiv (pl+l+2)F+2M$. By Lemma \ref{pullback} we know that $\Phi^{*}C$ and $\Phi^{*}F_{2}$ are reduced curves that has no irreducible components in common. In particular,
        $$
            \Delta_{\mathcal{G}} \geq \Phi^{*}C+\Phi^{*}F_{2}+F_{1}+F_{0}+M.
        $$
    in the other hand, $\Phi^{*}C \equiv \Phi^{*}(pF+M)\equiv plF+M$ and $\Phi^{*}F_{2} = lF$ and so
        $$
            \Phi^{*}C+\Phi^{*}F_{2}+F_{0}+F_{1}+M\equiv (pl+l+2)F+2M
        $$
    and by comparison degree we obtain the following equality
        $$
            \Delta_{\mathcal{G}} = \Phi^{*}C+\Phi^{*}F_{2}+F_{1}+F_{0}+M.
        $$
    In particular, $\Delta_{\mathcal{G}}$ is a reduced divisor and  $\mathcal{G}$ satisfies (\ref{tres}). We will show that 
    $\mathcal{G}$ satisfies (\ref{dois}). For that, let $Q$ the point in $F_{0}\cap M$. By the precedent case, more precisely, by
    (\ref{C}) we know that $Q$ is a $p$-reduced singularity. Fix $U$ an affine open set around  $Q = (0,0)$ such that $\Phi|_{\Phi^{-1}(U)}:\Phi^{-1}(U) \longrightarrow U$ is locally defined by $\Phi\colon (x,y)\mapsto (x^{l},y) = (\tilde{x},\tilde{y})$ and $\mathcal{F}$ is given by the $1$-form 
    $\omega = \alpha \tilde{y}d\tilde{x}-\tilde{x}d\tilde{y}+O(2)$. Then, $\Phi^{*}\mathcal{F}$ is given on $\Phi^{-1}(U)$ by the  $1$-form
        $
            \sigma = l\alpha ydx-xdy+O(2).
        $    
    Since we are assuming that $Q$ is a $p$-reduced singularity $\alpha \not\in \mathbb{F}_{p}$ and so $l\alpha \not\in \mathbb{F}_p$. We conclude that $\mathcal{G}$ has $\{Q\} = F\cap M$ as $p$-reduced singularity. In particular, satisfies (\ref{dois}). This finishes the proof of the case $(d_1,d_2) = (l,0)$ with  
    $l>1$ and $p\nmid l$.
The cases above finish the proof of the proposition.
\end{proof}
\subsection{Case B}

In this subsection, we show that a generic foliation  $\mathcal{G}$ in $\Projumprod$ with cotangent divisor $K_{\mathcal{G}} \equiv d_1F+d_2M$ satisfies the following
\begin{enumerate}[(i)]
    \item \label{Um}$K_{\mathcal{G}}\equiv d_{1}F_{0}+d_{2}M_{0}$ where $F_{0} = \{x_{0}= 0\}$ and $M_{0} = \{y_{0}=0\}$,
    \item \label{Dois} $F_{0}$ and $M= \{y_{1}=0\}$ are $\mathcal{G}$-invariant curves with  $\{Q\} = F_{0}\cap M$ a $p$-reduced singularity of $\mathcal{G}$,
    \item \label{Tres} The $p$-divisor $\Delta_{\mathcal{G}}$ is reduced
\end{enumerate}
when  $(d_1,d_2) \in \{(l_1,l_2)\in \mathbb{Z}^2\mid l_1,l_2>0 \mbox{ and } p\nmid l_1l_2 \}$. We start with the case $d_1 = d_1 = 1$.
\begin{prop} \label{primeirocasop1p1} A generic foliation in $\Projumprod$ with cotangent divisor $K_{\mathcal{G}}\equiv F+M$ satisfies (\ref{Um}),(\ref{Dois}) and (\ref{Tres}).
\end{prop}
\begin{proof}  Indeed, in the open set $U_{11}$ consider the foliation defined by the $1$-form:
$
\omega = ydx-xdy+\tilde{b}(x,y)dx+\tilde{a}(x,y)dy.
$
In the proof of Theorem \ref{casoplano}, we saw that for a generic choice of  $\tilde{a},\tilde{b} \in \corpo[x,y]_{2}$ the $1$-form $\omega$ defines a foliation with reduced $p$-divisor in $\mathbb{A}_{\corpo}^{2}$ which is explicitly given by $
\Delta = l_{1}+l_{2}+l_{3}+C$ where $l_{1},l_{2},l_{3}$ are distinct lines that pass through point $(0,0)$ and $C$ is an irreducible curve that passes through point $(0,0)$ of degree $p$ and with multiplicity $p-1$ over $(0,0)$. Moreover, by the Remark \ref{proj} we can assume that $C$ does not passes through points $\{[0:1:0],[1:0:0]\}$. For simplicity, we will assume that $l_{1} = \{x=0\}$, $l_{2} = \{y=0\}$ and $l_{3} = \{ux+vy=0\}$ for some non zero constants $u,v \in \corpo$. In this case, we have $\tilde{a}(x,y) = xa(x,y)$ and $\tilde{b}(x,y) = yb(x,y)$ for some $a,b \in \corpo[x,y]_{1}$.

Using the isomorphism  $\phi\colon U_{11}\longrightarrow \mathbb{A}_{\corpo}^{2}$ which associates $([x_0:x_1],[y_0:y_1])\mapsto \left(\frac{x_0}{x_1}, \frac{y_0}{y_1}\right)$ 
we can compactify $\omega$ to a foliation in $\Projumprod$. In this case, the foliation obtained is explicitly given by the $1$-form: $
\Omega = A_{0}dx_{0}+A_{1}dx_{1}+B_{0}dy_{0}+B_{1}dy_{1}$
where
$$A_{0} = x_{1}y_{0}y_{1}(x_{1}y_{1}+b(x_{0}y_{1},x_{1}y_{0})),\qquad A_{1} = -x_{0}y_{0}y_{1}(x_{1}y_{1}+b(x_{0}y_{1},x_{1}y_{0})),
$$
$$
B_{0} = x_{0}x_{1}y_{1}(-x_{1}y_{1}+a(x_{0}y_{1},x_{1}y_{0})), \quad \mbox{and} \quad B_{1} = -x_{0}x_{1}y_{0}(-x_{1}y_{1}+a(x_{0}y_{1},x_{1}y_{0})).
$$
The projective $1$-form $\Omega$ defines
a foliation  $\mathcal{F}_{\Omega}$ in $\Projumprod$ such that $K_{\mathcal{F}_{\Omega}}\equiv F+M$ and so $\mathcal{F}_{\Omega}$ satisfies (\ref{Um}).
We claim that $\Delta_{\mathcal{F}_{\Omega}}$ 
is reduced. Indeed, note that
$$
\Delta_{\mathcal{F}_{\Omega}} \equiv pK_{\mathcal{F}_{\Omega}}+N_{\mathcal{F}_{\Omega}}\equiv p(F+M)+(3F+3M) = (p+3)(F+M).
$$
By the homogeneous equations of $\mathcal{F}_{\Omega}$, we see that $F_{0} = \{x_{0} = 0\}$, $F_{1} = \{x_{1} = 0\}$, $M_{0} = \{y_{0}=0\}$ and $M = \{y_{1}=0\}$ are $\mathcal{F}_{\Omega}$-invariant. Moreover, 
 bi-projetivizing the line $L$ via $\Phi$ in $\Projumprod$ we know by the Lemma \ref{compactificacao} that the obtained curve ,$\overline{L}$, is irreducible  $\mathcal{F}_{\Omega}$-invariant with $\overline{L}\equiv F_{0}+M_{0}.$ Since $C$ does not pass through the point $\{[0:1:0],[1:0:0]\}$ we can apply the Lemma  \ref{compactificacao} to conclude that the irreducible curve of degree $p$ in  $\mathbb{A}_{\corpo}^{2}$ projectivize to an irreducible curve, $\overline{C}$,  in $\Projumprod$ that is  $\mathcal{F}_{\Omega}$-invariant and such that: $\overline{C} \equiv p(F_{0}+M_{0})$. So, we obtain $
\Delta_{\mathcal{F}_{\Omega}} = F_{0}+F_{1}+M_{0}+M+\overline{L}+\overline{C}$ and we ensure that  $\Delta_{\mathcal{F}_{\Omega}}$ is a reduced divisor. This shows that $\mathcal{F}$ satisfies (\ref{Tres}).

Observe that restricting the foliation to the open set $U_{10} = \{y_{0}\neq 0\}\cap \{x_{1} \neq 0\}$ the foliation in $\mathbb{A}_{\corpo}^{2}$ is given by the $1$-form
$
\sigma = y_{1}(y_{1}+b(x_{0}y_{1},1))dx_{0}-x_{0}(-y_{1}+a(x_{0}y_{1},1))dy_{1}.
$
Since $a, b \in \corpo[x,y]_{1}$ are generic, we can assume that  $a(x_{0}y_{1},1) = a_{0}+O(2)$ and $b(x_{0}y_{1},1) = b_{0}+O(2)$ where $a_{0}/b_{0} \not\in \mathbb{F}_{p}$, so that we obtain a foliation in  $\Sigma_{0}$ which satisfies (\ref{Um}), (\ref{Dois}) and (\ref{Tres}) with $d_{1} = d_{2} = 1$.
\end{proof}

Using the precedent proposition we will consider the general case.

\begin{prop}  Let $d_{1},d_{2}\in \mathbb{Z}_{\geq0}$ such that $p\nmid d_{i}$,if $d_{i}\neq 0$. Suppose that
$$(d_1,d_2) \in \{(l_1,l_2)\in \mathbb{Z}^2\mid l_1,l_2>0 \mbox{ and } p\nmid l_1l_2 \}. 
$$
Then, a generic foliation in  $\Projumprod$ with cotangent divisor $K_{\mathcal{G}}\equiv d_{1}F+d_{2}M$ satisfies (\ref{Um}),(\ref{Dois}) and (\ref{Tres}).
\end{prop}

\begin{proof}  Let $l \in \mathbb{N}$ be a positive integer coprime to $p$ and consider the finite map:
$$\begin{aligned}
          \Phi_{l} \colon \Projumprod & \longrightarrow \Projumprod &  \\
        ([x_{0},x_{1}],[y_{0},y_{1}])  & \mapsto([x_{0}^{l}:x_{1}^{l}],[y_{0}:y_{1}]).
    \end{aligned}$$
Let $R_{\Phi_{l}}$ be the ramification divisor of $\Phi_{l}$ and $\mathcal{G}$ the foliation in $\Projumprod$ described by the saturation of the 
$1$-form $\Phi_{l}^{*}\Omega$, where $\Omega$ is the projectivization of the $1$-form in  $\mathbb{A}_{\corpo}^{2}$ given by
$
\sigma = y_{1}(y_{1}+b(x_{0}y_{1},1))dx_{0}-x_{0}(-y_{1}+a(x_{0}y_{1},1))dy_{1}
$
where $a, b \in \corpo[x,y]_{1}$ are generic with
$a(x_{0}y_{1},1) = a_{0}+O(2)$, $b(x_{0}y_{1},1) = b_{0}+O(2)$ and $a_{0}/b_{0} \not\in \mathbb{F}_{p}$. Recall that by Proposition \ref{primeirocasop1p1} we can assume that the foliation defined by $\Omega$, $\mathcal{F}_{\Omega}$, is not $p$-closed with  $p$-divisor given by
$
\Delta_{\mathcal{F}_{\Omega}} = F_0+F_1+M_0+M+\overline{L}+\overline{C} 
$
where $F_{0} = \{x_0 = 0\}$, $F_1 =\{x_1=0\}$, $M_0 = \{y_0=0\}$, $M = \{y_1=0\}$ and 
with $\overline{L}$, $\overline{C}$ irreducible curves such that
$
\overline{C} \equiv p(F_0+M_0) \quad \mbox{and} \quad \overline{L}\equiv F_0+M_0.
$
We have $R_{\Phi_{l}} = (l-1)F_{0}+(l-1)F_{1} \equiv 2(l-1)F_{0}$ and since $F_{0}$ and $F_{1}$ are $\mathcal{F}_{\Omega}$-invariant we conclude that $K_{\mathcal{G}} = \Phi^{*}K_{\mathcal{F}} = \Phi^{*}(F+M) = lF+M$. In explicit terms, the foliation $\mathcal{G}$ is defined by the $1$-form:
$
\gamma =  \Tilde{A}_{0}dx_{0}+\Tilde{A}_{1}dx_{1}+\Tilde{B}_{0}dy_{0}+\Tilde{B}_{1}dy_{1}
$
where
$$\Tilde{A}_{0} = lx_{1}y_{0}y_{1}(x_{1}^{l}y_{1}+b(x_{0}^{l}y_{1},x_{1}^{l}y_{0})),\qquad \Tilde{A}_{1} = -lx_{0}y_{0}y_{1}(x_{1}^{l}y_{1}+b(x_{0}^{l}y_{1},x_{1}^{l}y_{0})),
$$
$$
\Tilde{B}_{0} = x_{0}x_{1}y_{1}(-x_{1}^{l}y_{1}+a(x_{0}^{l}y_{1},x_{1}^{l}y_{0})), \quad \mbox{and} \quad
\Tilde{B}_{1} = -x_{0}x_{1}y_{0}(-x_{1}^{l}y_{1}+a(x_{0}^{l}y_{1},x_{1}^{l}y_{0})).
$$

We will show the following equality:
    $
        \Delta_{\mathcal{G}} = F_{0}+F_{1}+M_{0}+M+\Phi_{l}^{*}\overline{L}+\Phi_{l}^{*}\overline{C}.
    $
Note that the curves $\Phi_{l}^{*}\overline{L}$ and $\Phi_{l}^{*}\overline{C}$ are $\mathcal{G}$-invariant. Indeed, since $\mdc(l,p) = 1$ and $x_{0},x_{1}$ do not divide the equations which define $\overline{L}$ and $\overline{C}$, using the Lemma \ref{pullback} we conclude that $\Phi_{l}^{*}\overline{L}$ and $\Phi_{l}^{*}\overline{C}$ are reduced and have distinct irreducible components. In particular, follows that  $\Phi_{l}^{*}\overline{L}$ and $\Phi_{l}^{*}\overline{C}$ are $\mathcal{G}$-invariant (see Proposition \ref{inv}). Since $F_{0}$, $F_{1}$, $M_{0}$, $M$ are $\mathcal{G}$-invariant 
    $
        \Delta_{\mathcal{G}}\geq F_{0}+F_{1}+M_{0}+M+\Phi_{l}^{*}\overline{L}+\Phi_{l}^{*}\overline{C}.
    $
Note that $\Delta_{\mathcal{G}}$ is a divisor that has bi-degree given by the formula:
    $$ 
        pK_{\mathcal{G}}+N_{\mathcal{G}} \equiv p(lF_{0}+M_{0})+((l+2)F_{0}+3M_{0})\equiv (pl+l+2)F_{0}+(p+3)M_{0}.
    $$ 
In the other hand
$$F_{0}+F_{1}+M_{0}+M+\Phi_{l}^{*}\overline{L}+\Phi_{l}^{*}\overline{C} \equiv F_{0}+F_{0}+M_{0}+M_{0}+(lF_{0}+M_{0})+(plF_{0}+pM_{0})$$
$$
= (pl+l+2)F_{0}+(p+3)M_{0}.$$
So, by bi-degree comparison,  we conclude the following equality
$
\Delta_{\mathcal{G}} = F_{0}+F_{1}+M_{0}+M+\Phi_{l}^{*}\overline{L}+\Phi_{l}^{*}\overline{C}.
$
This shows that $\mathcal{G}$ satisfies (\ref{Tres}). 

We need to show that $F_0\cap M  = \{x_{0} = 0\}\cap \{y_{1}=0\}$ is a $p$-reduced singularity of $\mathcal{G}$.  Indeed, restricting the foliation $\mathcal{G}$ to the open set $U_{10}$ we obtain a foliation in $\mathbb{A}_{\corpo}^{2}$ given by the $1$-form
$
\omega = ly(y+b(x^{l}y,1))dx-x(-y+a(x^{l}y,1))dy
$
and we conclude that $(0,0)$ is $p$-reduced since we are assuming that $a(x_{0}y_{1},1) = a_{0}+O(2)$ and $b(x_{0}y_{1},1) = b_{0}+O(2)$ with $a_{0}/b_{0}\not\in \mathbb{F}_{p}$. By symmetry, that is, replacing $F_{0}$ by $M_{0}$,
follows that given $d_{1},d_{2}\in \mathbb{Z}_{\geq 0}$ with $p\nmid d_{i}$ (if $d_{i}\neq 0$) we can construct foliations in 
$\Projumprod$ with the following properties:

\begin{enumerate}[(i)]

    \item The canonical divisor $K_{\mathcal{G}}$ is numerically equivalent to the divisor  $d_{1}F_{0}+d_{2}M_{0}$ where $F_0$ and $M_0$ are fibers of  the first and second projection, respectively;
 
    \item The curves $F_{0}$ and $M_0$ are $\mathcal{G}$-invariant and  $\{Q\} = F_{0}\cap M_0$ is a $p$-reduced singularity of $\mathcal{G}$;
    
    \item The $p$-divisor $\Delta_{\mathcal{G}}$ is reduced.
\end{enumerate}
This finishes the proof of the proposition.
\end{proof}

Considering the junction of the case A and B we obtain the complete proof of the Theorem \ref{casop1p1}.

\section{The \texorpdfstring{$p$}{p}-divisor for foliations on Hirzebruch surfaces} \label{pdivisorSigma}

Let $d\in \mathbb{Z}_{\geq 0}$ and $M_{d}$ be a section of the  natural projection $\pi: \Sigma_{d}\longrightarrow \Projum$ and $F$ be a fiber of $\pi$ such that $M_{d}^{2} = d$ and $M_{d}\cdot F = 1$. The curves $M_{d}$ and $F$ form a base of the vector space $\Num_{\mathbb{Q}}(\Sigma_{d})$. Note that for any other divisor
$D \in \Div(\Sigma_{d})$ such that $\langle D, F \rangle=  \Num_{\mathbb{Q}}(\Sigma_{d})$ and which satisfies $D^{2} = d$ and $D\cdot F = 1$ we have $D \equiv M_{d}$. Indeed, write
$D = aM_{d}+bF$ for some $a,b \in \mathbb{Q}$ and observe that $1= D\cdot F = a$ and $d = D^{2} = a^{2}d+2ab = d+2b$ and so $b = 0$.

Fix $\{F,M_{d}\}$ as a basis for $\Num_{\mathbb{Q}}(\Sigma_{d})$ and let $Q$ be the point in the intersection $F\cap M$, where $M=\{y_1=0\}$ is the curve in $\Sigma_{d}$ with negative self-intersection, that is, $M^{2} = -d$. Let
$Bl_{Q}\colon Bl_{Q}(\Sigma_{d}) \longrightarrow  \Sigma_{d}$ the blowup of $\Sigma_{d}$ at $Q$ and
$
c_{\Tilde{F}}\colon Bl_{Q}(\Sigma_{d}) \longrightarrow \Sigma_{d+1} 
$
the map that consists in the contraction of the strict transform of 
$F$ in $Bl_{Q}(\Sigma_{d})$. Let $\overline{E}$ the exceptional divisor associated to $Bl_{Q}$ and 
$\Phi_{d}^{d+1}\colon \Sigma_{d+1}-\!\! \rightarrow \Sigma_{d}$
the rational map which consist in the composition of birational maps:
$$
\Phi_{d}^{d+1} = Bl_{Q}\circ c_{\Tilde{F}}^{-1}: \Sigma_{d+1}-\!\! \rightarrow Bl_{Q}(\Sigma_{d}) \longrightarrow \Sigma_{d}. 
$$
Let $\Tilde{M}_{d+1}$ the strict transform of  $M_{d}$ in $Bl_{Q}(\Sigma_{d})$ and consider $E = (c_{\Tilde{F}})_{*}\overline{E}$ and $M_{d+1} = (c_{\Tilde{F}})_{*}\Tilde{M}_{d+1}$ the induced curves by the contraction. Observe that $\{E, M_{d+1}\}$ forms a basis for $\Num_{\mathbb{Q}}(\Sigma_{d})$ and satisfies the following conditions:
$E^{2} = 0$, $E\cdot M_{d+1} = 1$ and $M_{d+1}^{2} = d+1$.

\begin{lemma}\label{elementar}  If $\mathcal{F}$ 
is a foliation in $\Sigma_{d}$ with $K_{\mathcal{F}} \equiv d_{1}F+d_{2}M_{d}$ then $\mathcal{G}=(\Phi_{d}^{d+1})^{*}\mathcal{F}$ is a foliation in $\Sigma_{d+1}$ with $
K_{\mathcal{G}} \equiv (d_{1}-l(Q)+1)E+d_{2}M_{d+1}.$
\end{lemma}

\begin{proof} Let $\mathcal{H} = Bl_{Q}^{*}\mathcal{F}$ be the induced foliation in $Bl_{Q}(\Sigma_{d})$. Then, 
$K_{\mathcal{G}} = (c_{\Tilde{F}})_{*}K_{\mathcal{H}}$ and 
$$
K_{\mathcal{H}} = Bl_{Q}^{*}K_{\mathcal{F}}+(1-l(Q))E \equiv Bl_{Q}^{*}(d_{1}F+d_{2}M_{d})+(1-l(Q))\overline{E}.
$$
Since $Bl_{Q}^{*}F = \Tilde{F}+\overline{E}$ and $Bl_{Q}^{*}M_{d} = \Tilde{M}_{d+1}$, we conclude that
$
K_{\mathcal{H}} \equiv d_{1}\Tilde{F}+d_{2}\Tilde{M}_{d+1}+(d_{1}-l(Q)+1)\overline{E}
$
and it follows
$$
K_{\mathcal{G}} = (c_{\Tilde{F}})_{*}K_{\mathcal{H}} \equiv (d_{1}-l(Q)+1)(c_{\Tilde{F}})_{*}\overline{E}+d_{2}(c_{\Tilde{F}})_{*}\Tilde{M}_{d+1} \equiv (d_{1}-l(Q)+1)E+d_{2}M_{d+1}
$$
which finishes the proof of the lemma.
\end{proof}

\begin{thmB} Let $\Sigma_{d}$ the $d$-Hirzebruch surface over $\corpo$ and $d_{1},d_{2} \in \mathbb{Z}_{\geq 0}$ such that $p\nmid d_{i}$, if $d_i\neq 0$. Then, a generic foliation in $\Sigma_{d}$ with normal bundle  $N\equiv (d_{1}-d+2)F+(d_{2}+2)M_{d}$ has reduced $p$-divisor. More precisely, $U(\Sigma_{d},\mathcal{O}_{\Sigma_{d}}(F+M_{d}),N)\neq \emptyset$ if
\begin{itemize}
    \item $N\cdot F-2\geq 0$ and $p\nmid N\cdot F-2$ (if nonzero);
    \item  $\deg(N)-(1+d)N\cdot F+d-2 \geq 0$ and $p\nmid \deg(N)-(1+d)N\cdot F+d-2$ (if nonzero).
\end{itemize}
\end{thmB}

\begin{proof} The proof will be done by induction on $d$. The case $d = 0$ was considerate in the Theorem \ref{casop1p1}. Suppose that the result is true for all $j$-Hirzebruch surfaces with $j\leq d$, that is, suppose that
for each $j\in\{0,\ldots,d\}$ and for each $d_{1},d_{2} \in \mathbb{Z}_{\geq 0}$ with $p\nmid d_{i}$ (if $d_{i}\neq 0$) 
we can find a foliation $\mathcal{H}$ such that
\begin{enumerate}[(i)]
    \item \label{uum} $K_{\mathcal{H}}\equiv d_{1}F_0+d_{2}M_{j}$, where $F_0$ is a fiber of the natural projection $\pi\colon \Sigma_{j}\longrightarrow \Projum$ and $M_{j}$ is a section of $\pi$ such that $M_{j}^{2} = j$ and $M_{j}\cdot F = 1$;
    \item \label{ddois} If $M$ is the section of  negative self-intersection of $\Sigma_{j}$ then $F_0$ and $M$ are $\mathcal{H}$-invariant and $\{Q\} = F_{0}\cap M$ is a $p$-reduced singularity of  $\mathcal{H}$ where $M^{2} = -j$;
    \item \label{ttres} The $p$-divisor $\Delta_{\mathcal{H}}$ is reduced.
\end{enumerate}
We will show that the same holds over $\Sigma_{d+1}$. Fix $d_{1}, d_{2} \in \mathbb{Z}_{\geq 0}$ such that $p\nmid d_{i}$, if $d_{i}\neq 0$. Let $\mathcal{G}$ a foliation in $\Sigma_{d}$ which satisfies (\ref{uum}),(\ref{ddois}) and (\ref{ttres}) for $j=d$.

Let $Q$ the $p$-reduced singularity in $F_0\cap M$ and $\pi_{Q}\colon Bl_{Q}(\Sigma_{d}) \longrightarrow \Sigma_{d}$ the blowup with center at $Q$. By the Lemma \ref{explosao}, we know that $Bl_{Q}^{*}\mathcal{F}$ is a foliation in $Bl_{Q}(\Sigma_{d})$ with reduced $p$-divisor. More precisely, if $\overline{E}$ denotes the exceptional divisor
and $\mathcal{G} = \pi_{Q}^{*}\mathcal{F}$, we know that
$
\Delta_{\mathcal{G}} = \Tilde{\Delta}_{\mathcal{F}}+\overline{E}.
$
We claim that since $Q$ is $p$-reduced we ensure that over $\overline{E}$ there exists at least a $p$-reduced singularity of $\mathcal{G}$. 

Indeed, in an open neighborhood of $Q$ the foliation is given by a $1$-form of the type$
\omega = \alpha ydx+xdy+O(2)$
and considering the chart $\pi_{Q}\colon (x,t)\mapsto (x,xt)$ we see that in a neighborhood of $\overline{E}$ the foliation $\mathcal{G}$ is given by the $1$-form:$
\sigma =(\alpha+1)tdx+xdt+O(2).$
In particular, $Q$ is $p$-reduced. Denote by $c_{F}\colon Bl_{Q}(\Sigma_{d}) \longrightarrow Y$ the contraction of the line $\Tilde{F}$ that  is the strict transform of $F_0$. Since $Q\in F_0$, the line $\Tilde{F}$ has self-intersection
$-1$ so that there exists that contraction. The out surface, $Y$,  is the Hirzebruch surface of type $d+1$ ($\Sigma_{d+1}$). We will show that $c_{F_{*}}\mathcal{G}$ is a foliation in $\Sigma_{d+1}$ that satisfies (\ref{uum}), (\ref{ddois}) and (\ref{ttres}). Let $\overline{E}$ the exceptional divisor associated with the blowup $Bl_{Q}$ and $\Tilde{M}_{d+1}$ the strict transform  of the curve  $M_{d}$. Denote by $E = c_{F_{*}}\overline{E}$ and $M_{d+1} = c_{F_{*}}\tilde{M_{d+1}}$ the induced curves by the contraction $c_{F}$. Observe that we have the following formulas: $E^{2} = 0$,$E\cdot M_{d+1} = 1$ and $M_{d+1}^{2} = d+1.$ If $\{P\} = c_{\Tilde{F}}(\Tilde{F})$ then $c_{\Tilde{F}}\colon Bl_{Q}(\Sigma_{d})-\{\Tilde{F}\} \longrightarrow \Sigma_{d+1}-\{P\}$ is an isomorphism. In particular, as $\Delta_{\mathcal{F}}$ is reduced we have
$\Delta_{c_{\Tilde{F}_{*}}\mathcal{G}}$ is reduced and so that  we have (\ref{ttres}). The local verification done above shows that $E\cap M = \{Q_{2}\}$ is a $p$-reduced singularity of  $c_{\Tilde{F}_{*}}\mathcal{G}$ and the Lemma \ref{elementar} ensures that $K_{\mathcal{G}}\equiv d_{1}E+d_{2}M_{d+1}$. So, we obtains (\ref{uum}). This finishes the proof of the theorem.
\end{proof}

\begin{OBS} Let $d \in \mathbb{Z}_{\geq 0}$ and $F$ be a section of the natural projection $\pi\colon\Sigma_{d}\longrightarrow \Projum$ and $M_{d}$ a section of $\pi$ such that $F\cdot M_{d} = 1$ and $M_{d}^{2} = d.$ Define
$$S_{d} = \{(d_{1},d_{2})\in \mathbb{Z}^{2}\mid \mbox{there exists a foliation in $\Sigma_{d}$ with canonical divisor $K\equiv d_{1}F+d_{2}M_{d}$}\}.$$

Follows from \cite[Proposition 3.6]{galindo2020foliations} that:
\begin{itemize}
    \item $S_{0} = \{(d_{1},d_{2})\in \mathbb{Z}^{2}\mid d_{1},d_{2}\geq 0\}\cup\{(-2,0)\}\cup\{(0,-2)\}$,


    \item If $d>0$: $S_{d} = \{(d_{1},d_{2})\in \mathbb{Z}^{2}\mid d_{1}\geq -1,d_{2}\geq 0\}\cup \{(d,-2)\}.$ 
\end{itemize}
Now, define
$$
R_{d} = \{(d_{1},d_{2})\in \mathbb{Z}^{2}\mid \mbox{there exists a foliation in $\Sigma_{d}$ with reduced $p$-divisor and $K\equiv d_{1}F+d_{2}M_{d}$}\}.
$$
Theorem \ref{Hirz} can be formulated in the following way:
\begin{itemize}
    \item $R_{0} = S_{0}-S_{0}(p)\cup\{(-2,0)\}\cup\{(0,-2)\},$
    \item Se $d>0$: $R_{d} = S_{d}-S_{d}(p)\cup\{(-1,d)\mid d \in \mathbb{Z}_{\geq0}\}\cup \{(d,-2)\}$
\end{itemize}
where $S_{d}(p)= \{(d_{1},d_{2})\in S_{d}\mid p\mid d_{1}\mbox{ or } p\mid d_{2}\}-\{(0,0)\}$. 
\end{OBS}

The following proposition shows that for each point $P$ on $l = \{(-1,e) \in \mathbb{Z}^2\mid e \in \mathbb{Z}_{\geq 0}\}$ we can find an example of foliation such that their $p$-divisor has a $p$-factor.

\begin{prop} Let $\mathcal{F}$ be a foliation that is not $p$-closed of degree $d$ in $\Projdois$ and $Q \in \Projdois-\sing(\mathcal{F})\cup \mathcal{Z}(\Delta_{\mathcal{F}})$. Let $\mathcal{G} = Bl_{Q}^{*}\mathcal{F}$ the induced foliation in $Bl_{Q}(\Projdois)$. Then,
$\mathcal{G}$ is not $p$-closed and $\Delta_{\mathcal{G}}$ has a  $p$-factor. Moreover,
$
K_{\mathcal{G}} \equiv -F+dM_{1}.
$
\end{prop}
\begin{proof} The condition $Q \not\in \sing(\mathcal{F})$ implies that $l(Q) = 0$. So, we have $N_{\mathcal{G}} = Bl_{Q}^{*}N_{\mathcal{F}}$ and $K_{\mathcal{G}} = Bl_{Q}^{*}K_{\mathcal{F}}+(1-l(Q))E = Bl_{Q}^{*}K_{\mathcal{F}}+E.$ So,
$
[\Delta_{\mathcal{G}}] = pK_{\mathcal{G}}+N_{\mathcal{G}} =  Bl_{Q}^{*}\Delta_{\mathcal{F}}+pE = \Tilde{\Delta}_{\mathcal{F}}+pE
$
where the last equality follows from the condition $Q \not\in \mathcal{Z}(\Delta_{\mathcal{F}})$. In particular, $\od_{E}(\Delta_{\mathcal{G}}) = p.$ Now, we have
$
K_{\mathcal{G}}= Bl_{Q}^{*}K_{\mathcal{F}}+E \equiv (d-1)(E+F)+E\equiv (d-1)F+dE.
$
Write $M_{1} \equiv aF+bE$. Since $1 = M_{1}\cdot F = b$ we have $M_{1} \equiv aF+E$ and since $1 = M_{1}^{2} = (aF+E)^{2} = 2a-1$ we obtain $a = 1$. So, $M_{1} = F+E$ and this implies $E = M_{1}-F$. So,
    $
        K_{\mathcal{G}} \equiv (d-1)F+dE \equiv (d-1)F+d(M_{1}-F) \equiv -F+dM_{1}
    $
which ends the proof of the proposition.
\end{proof}

\normalfont \noindent \textbf{Acknowledgements.} W.Mendson thanks J.V. Pereira for discussions about the results of this paper and for their remarks about this work. The author also thanks Eduardo Vital for suggestions and comments on the manuscript and also thanks Cesar Hilario for comments on the introduction. The author acknowledge support of CNPq, Faperj, and the Instituto de Matemática Pura e Aplica (IMPA).

\nocite{MR1440180}\nocite{MR3687427}\nocite{galindo2020foliations}\nocite{MR3328860}\nocite{mendson2022}

\nocite{MR1440180}\nocite{MR3687427}\nocite{galindo2020foliations}\nocite{MR3328860}

\bibliographystyle{siam}

\bibliography{annot}

\end{document}